\documentclass[11pt]{article}

\usepackage{amsmath}
\usepackage{amsfonts}
\usepackage{amsbsy}
\usepackage{amsthm}
\usepackage{amssymb}
\usepackage{graphicx}
\usepackage{psfrag}
\usepackage{epsfig}
\usepackage{citesort}
\usepackage{graphicx}
\usepackage{citesort}

\newtheorem{lemma}{Lemma}

\theoremstyle{remark}
\newtheorem{remark}{Remark}

\begin{document}

\title{A high-order integral solver for
  scalar problems of diffraction by screens and apertures in three
  dimensional space}

\author{Oscar P. Bruno and St\'ephane K. Lintner}

\maketitle
\begin{abstract}
  We present a novel methodology for the numerical solution of problems of
  diffraction by infinitely thin screens in three dimensional space. Our
  approach relies on new integral formulations as well as associated
  high-order quadrature rules. The new integral formulations involve {\em
  weighted versions} of the classical integral operators associated with the
  thin-screen Dirichlet and Neumann problems as well as a generalization to
  the open surface problem of the classical Calder\'on formulae.  The
  high-order quadrature rules we introduce for these operators, in turn,
  resolve the multiple Green function and edge singularities (which occur at
  arbitrarily close distances from each other, and which include weakly
  singular as well as \emph{hypersingular} kernels) and thus give rise to
  super-algebraically fast convergence as the discretization sizes are
  increased. When used in conjunction with Krylov-subspace linear algebra
  solvers such as GMRES, the resulting solvers produce results of high
  accuracy in small numbers of iterations for low and high frequencies
  alike. We demonstrate our methodology with a variety of numerical results
  for screen and aperture problems at high frequencies---including
  simulation of classical experiments such as the diffraction by a circular
  disc (including observation of the famous Poisson spot), interference
  fringes resulting from diffraction across two nearby circular apertures,
  as well as more complex geometries consisting of multiple scatterers and
  cavities.
\end{abstract}

\section{Introduction\label{intro}} 
Diffraction problems involving \textit{infinitely thin screens} play central
roles in the field of wave propagation: as a noted early example we mention
the experimental observation of a bright area in the shadow of the disc (the
famous Poisson spot), which provided one of the earliest confirmations of
the wave-theory models of light~\cite[p. xxviii]{Wolf}. Certainly, problems
of diffraction by screens continue to impact significantly on a varied range
of present day technologies, including wireless communications, electronics
and photonics, as well as sonic metamaterials, sound transmission,
non-invasive evaluation and environmental acoustics. Like other wave
scattering problems, screen problems (mathematically described by
\emph{open} surface problems~\cite{Stephan}) can be treated by means of
numerical methods that rely on approximation of the equations of
electromagnetism and acoustics over volumetric domains (on the basis of,
e.g., finite-difference or finite-element methods) as well as methods based
on boundary integral equations. Boundary-integral methods do not suffer from
the well known pollution/dispersion
errors~\cite{BabuskaSauterStefan,Jameson}, they require discretization of
domains of lower dimensionality than those involved in volumetric methods,
and, in spite of the fact that they give rise to full matrices, they can be
treated efficiently, even for high-frequencies, by means of accelerated
iterative scattering
solvers~\cite{BleszynskiBleszynskiJaroszewicz,BrunoKunyansky,Rokhlin,BrunoEllingTurc}.
Unfortunately, integral methods for open surfaces do not give rise, at least
in their classical formulations, to Fredholm integral operators of the
second-kind, and they suffer, like their volumetric counterparts, from
singularities in the vicinity of open edges.  As a result of these
characteristics, both volumetric and boundary-integral numerical methods for
open surface problems have typically proven computationally expensive and
inaccurate.

Generalizing the two-dimensional open-arc method introduced
in~\cite{LintnerBruno,BrunoLintner2}, this paper presents a novel approach
for the treatment of open surface diffraction problems in three-dimensional
space, the benefits of which are two-fold. On one hand, the new method
enjoys high-order accuracy: by considering weighted versions
$\mathbf{S}_\omega$ and $\mathbf{N}_\omega$ of the single-layer operator
$\mathbf{S}$ and hypersingular operator $\mathbf{N}$, and thus explicitly
extracting the solutions' edge singularity, we introduce high-order
quadrature rules (based on the partition of unity
method~\cite{BrunoKunyansky} and a combination of polar and quadratic
changes of variables) which acurately resolve the multiple Green function
and edge singularities that occur at arbitrarily close distances from each
other, and which include weakly singular as well as \emph{hypersingular}
kernels.  On the other hand, the method gives rise to well-behaved linear
algebra: as shown in this paper, the composite operator
$\mathbf{N}_\omega\mathbf{S}_\omega$ (which in the two-dimensional case was
rigorously proven~\cite{LintnerBruno,BrunoLintner2} to be a second-kind
Fredholm operator) requires very small number of iterations when used in
conjunction with the linear iterative solver GMRES. In particular, the
computational times required by our non-accelerated open surface solvers are
comparable to those required by the corresponding non-accelerated version of
the closed surface solver presented in~\cite{BrunoKunyansky}. (An extension
of the acceleration method~\cite{BrunoKunyansky} to the present context,
which is not pursued here, does not present difficulties.) Thus, the present
methodologies enable solution of the classical open surface scattering
problems with an efficiency and accuracy comparable to that available for
the smooth closed-surface counterparts.

The difficulties associated with boundary-integral methods for open surface
problems are of course well-known, and significant efforts have been devoted
to their treatment. The contributions
\cite{PovznerSuharesvki,ChristiansenNedelec} sought to generalize the
Calder\'on relations in the open-surface context as a means to derive
second-kind Fredholm equations for these
problems. In~\cite{PovznerSuharesvki} it was shown that the combination
$\mathbf{N}\mathbf{S}$ can be expressed in the form
$\mathbf{I}+\mathbf{T}_K$, where the kernel $\mathbf{K}(x,y)$ of the
operator $\mathbf{T}_K$ has a \textit{polar} singularity of the type
$O\left(\frac{1}{|x-y|}\right)$. This result is not uniform throughout the
surface, and it does not take into account the singular edge behavior: the
resulting operator $\mathbf{T}_K$ is not compact (in fact it gives rise to
strong singularities at the surface edge~\cite{LintnerBruno}), and the
operator $\mathbf{I} + \mathbf{T}_K$ is therefore not a second-kind operator
in any meaningful functional space. When used in conjunction with boundary
elements that vanish at edges, however, the combination
$\mathbf{N}\mathbf{S}$ can give rise to reduction of iteration numbers, as
demonstrated in reference~\cite{ChristiansenNedelec} through numerical
examples concerning low-frequency problems. The
contribution~\cite{ChristiansenNedelec} does not include details on
accuracy, and it does not utilize integral weights to resolve the solution's
edge singularity.  A related but different method was introduced
in~\cite{AntoineBendaliDarbasmarion} which, once again, exhibits small
iteration numbers at low frequencies, but which does not resolve the
singular edge behavior and for which no accuracy studies have been
presented.  An effective approach for regularization of the singular edge
behavior in the two-dimensional open-arc problem is based on use of a cosine
change of variables; see~\cite{BrunoLintner2} and references therein. In the
three dimensional case under consideration in this paper, high-order
integration rules for the single-layer and hypersingular operators were
introduced in~\cite{Stephan}, but these methods have only been applied to
problems of low frequency, and they have not been used in conjunction with
iterative solvers.

The remainder of this paper is organized as
follows. Section~\ref{sec_prelim} defines the Dirichlet and Neumann problem
on an open surface and it briefly discusses a number of difficulties
inherent in classical formulations; Section~\ref{sec_formulation} introduces
the new weighted operators; and Section~\ref{numer_framework} provides an
outline of the Nystr\"om-based numerical framework on which the solvers are
based.  The next five sections describe the construction of the high-order
numerical approximations we use for weighted operators: Sections~\ref{sec_S}
through~\ref{sec_combined} decompose the operators into six \emph{canonical
integral types}, while Sections~\ref{sec_interior} and~\ref{sec_exterior}
provide high-order integration rules for each one of the canonical
operators.  The selection of certain parameters required by our solvers are
detailed in Section~\ref{param_sel}. Finally, numerical results are
presented in Section~\ref{sec_numerical} which demonstrate the properties of
the integral formulations and solvers introduced in this paper across a
range of frequencies and geometries---including simulation of classical
experiments such as the diffraction by a circular disc (including
observation of the famous Poisson spot), interference fringes resulting from
diffraction across two nearby circular apertures, as well as more complex
geometries consisting of multiple scatterers and cavities.

\section{Open-Surface Acoustic Diffraction Problems}\label{sec_prelim}
Throughout this paper $\Gamma$ denotes a \textit{smooth} open surface (also
called a screen~\cite{Stephan}) with a smooth edge $\partial \Gamma$ in
three-dimensional space.
\subsection{Classical integral equations}
We consider the sound-soft and sound-hard problems of acoustic scattering by the open screen $\Gamma$, that is, the Dirichlet and Neumann boundary value problems
\begin{equation}\label{b_conds}
\left\{ \begin{array}{llll} \Delta u +k^2 u = 0 \quad\mbox{outside}\quad
  \Gamma , & u|_{\Gamma} = f, & \mbox{(sound-soft)},\\
\Delta v +k^2 v = 0 \quad\mbox{outside}\quad \Gamma, & \frac{\partial
  v}{\partial n}|_{\Gamma} = g, & \mbox{(sound-hard)},
\end{array}\right.
\end{equation}
 for the Helmholtz equation, where $u$ and $v$ are radiating functions at
infinity.

As is well known, both boundary-value problems are uniquely solvable in adequate functional spaces~\cite{Stephan}. For $\mathbf{r}$ outside $\Gamma$, the solution of the Dirichlet and Neumann problems can be expressed as a single-layer potential
 \begin{equation}\label{single_pot}
 u(\mathbf{r})= \int_\Gamma G_k(\mathbf{r},\mathbf{r}')\mu(\mathbf{r}')
 dS'
 \end{equation}
and a double-layer potential
 \begin{equation}
 v(\mathbf{r})= \int_\Gamma \frac{\partial
 G_k(\mathbf{r},\mathbf{r}')}{\partial
 \textbf{n}_\mathbf{r}'}\nu(\mathbf{r}') dS' ,
 \end{equation}
respectively, where $G_k$ denotes the free-space Green's function
\begin{equation}\label{Green}
G_k(\mathbf{r},\mathbf{r}')=\frac{e^{ik|\mathbf{r}-\mathbf{r}'|}}{|\mathbf{r}-\mathbf{r}'|},\quad \mathbf{r} \neq \mathbf{r}'.
\end{equation}
Letting $\mathbf{S}$ and $\mathbf{N}$ denote the classical single-layer and hypersingular operators
\begin{equation}\label{Sdef}
 \mathbf{S}[\mu](\mathbf{r})\equiv \int_\Gamma
 G_k(\mathbf{r},\mathbf{r}')\mu(\mathbf{r}) dS', \textbf{ }\quad \mathbf{r}\mbox 
 { on } \Gamma
 \end{equation}
and
 \begin{equation}\label{Ndef}
 \mathbf{N}[\nu](\mathbf{r})\equiv \lim\limits_{z\rightarrow 0} \frac{\partial
 }{\partial \textbf{n}_{\mathbf{r}}}\int_\Gamma \frac{\partial
 G_k(\mathbf{r},\mathbf{r}'+z\textbf{n}_{\mathbf{r}'})}{\partial
 \textbf{n}_\mathbf{r}'}\nu(\mathbf{r}') dS',\textbf{ }\quad \mathbf{r}\mbox 
 { on } \Gamma
 \end{equation}
the densities $\mu$ and $\nu$ are the unique solutions of the integral
equations
\begin{equation}\label{Sbad}
\mathbf{S}[\mu]=f\quad\mbox{and}
\end{equation}
\begin{equation}\label{Nbad}
\quad \mathbf{N}[\nu]=g.
\end{equation}
The integral operators in~\eqref{Sbad} and~\eqref{Nbad} have eigenvalues which accumulate at
zero and infinity respectively, and, thus, solutions of~\eqref{Sbad} and~\eqref{Nbad} by means
of Krylov-subspace iterative solvers such as GMRES generally require large
number of iterations. Furthermore, as discussed in Section~\ref{sec_sing},
the solutions $\mu$ and $\nu$ are singular at the edge of $\Gamma$ and thus
give rise to low order convergence unless such singularities are
appropriately taken into account. 

\subsection{Calder\'on formulation in the case of closed surfaces and shortcomings in a direct extension to open surfaces}\label{sec_calderon}
In the case where the surface $\Gamma_c$ under consideration is closed (that is, it equals the boundary of a bounded set in space), second kind Fredholm equations can be
derived either by making use of the classical jump relations
across the surface of the double-layer potential or the normal derivative of
the single-layer potential~\cite{ColtonKress2}, or, alternatively, by
relying on the Calder\'on formula which establishes that the combination
$\mathbf{N}_c \mathbf{S}_c$ of the closed-surface hypersingular operator
$\mathbf{N}_c$ and single-layer operator $\mathbf{S}_c$ can be expressed in
the form
\begin{equation}\label{CalderonClosed}
\mathbf{N}_c\mathbf{S}_c=-\frac{I}{4}+\mathbf{K}_c,
\end{equation}
where $\mathbf{K}_c$ is a compact operator in a suitable function
space~\cite{Nedelec}.

In the case of an open surface, the requirement~\eqref{b_conds} that the
same limit be achieved on both sides of the surface prevents the use of
discontinuous potentials. And, use the Calder\'on
formula~\eqref{CalderonClosed} does not give rise to a Fredholm equation in
the function spaces associated with open-screen problems: for example, the
composition of $\mathbf{N}$ and $\mathbf{S}$ is not even defined in the
functional framework set forth in~\cite{Stephan}---since, as
shown in~\cite{BrunoLintner2,LintnerBruno}, the image of the operator
$\mathbf{S}$ (the Sobolev space $H^\frac{1}{2}(\Gamma)$) is larger than the
domain of definition of $\mathbf{N}$ (the Sobolev space
$\tilde{H}^{\frac{1}{2}}(\Gamma)$). It is interesting to note, further,
that, as shown in~\cite{LintnerBruno} in the two-dimensional case, the image
of a constant function has a strong edge singularity,
\begin{equation}\label{NSblow}
\mathbf{N}\mathbf{S}[1](\mathbf{r})= O(\frac{1}{d(\mathbf{r})}),
\end{equation}
where $d$ denotes the distance to the edge---which demonstrates the
degenerate character of the composite operator $\mathbf{N}\mathbf{S}$.

\subsection{Singular edge-behavior}\label{sec_sing}
Although issues related to the singularity of the solutions $\mu$ and $\nu$
of equations~\eqref{Sbad} and~\eqref{Nbad} were controversial at
times~\cite{Bethe,BouwkampOnBethe}, the singular character of these
solutions is now well
documented~\cite{Meixner,Maue,BouwkampReview,Stephan,Costabel}. In
particular, in reference~\cite{Costabel} it was shown that $\mu$ and $\nu$
can be expressed in the forms
\begin{equation}\label{CostabelExp}
\mu\sim\frac{\varphi}{\sqrt{d}},\quad\nu\sim\psi\sqrt{d},
\end{equation}
where $\varphi$ and $\psi$ are infinitely differentiable functions throughout
$\Gamma$, \textit{up to and including the edge}. Thus the singular behavior
of the solutions to~\eqref{Sbad} and~\eqref{Nbad} is fully characterized by the factors
$d^{1/2}$ and $d^{-1/2}$ in equation~\eqref{CostabelExp}.

\section{Weighted operators and regularized formulation\label{sec_formulation}}
In view of the regularity results~\eqref{CostabelExp} we introduce a weight
$\omega(\mathbf{r})$ which is \textit{smooth, positive and non vanishing} across the interior of the surface, and which has square-root asymptotic edge behavior:
\begin{equation}\label{asymp}
 \omega \sim d^{1/2}.
\end{equation} 
We then define the weighted operators
\begin{equation}\label{Somega}
\mathbf{S}_\omega[\varphi]=  \mathbf{S}\left[\frac{\varphi}{\omega}\right],\quad\mathbf{N}_\omega[\psi] = \mathbf{N}[\omega\psi],
\end{equation}
so that for functions $f$ and $g$ that are smooth on $\Gamma$, up to and including the edge $\partial \Gamma$, the solutions of the equations
\begin{equation}\label{SGood}
\mathbf{S}_\omega[\varphi]= f,
\end{equation}
and
\begin{equation}\label{NGood}
\quad \mathbf{N}_\omega[\psi]= g
\end{equation}
are also smooth throughout the surface. In view of the closed-surface
Calder\'on formula~\eqref{CalderonClosed}, we consider the combined operator
$\mathbf{N}_\omega\mathbf{S}_\omega$ and the corresponding equations
\begin{equation}\label{NSGoodD}
\mathbf{N}_\omega\mathbf{S}_\omega[\varphi] =
\mathbf{N}_\omega[f]\quad\mbox{and}
\end{equation}
\begin{equation}\label{NSGoodN}
\mathbf{N}_\omega\mathbf{S}_\omega[\psi] = g.
\end{equation}

Note that the solutions of equations~\eqref{SGood} and~\eqref{NGood} are
related to those of~\eqref{NSGoodD} and~\eqref{NSGoodN}:
\begin{equation} 
\mu=\frac{\varphi}{\omega}, \quad \nu = \omega \cdot \mathbf{S}_\omega[ \psi ].
\end{equation}
As shown in~\cite{LintnerBruno,BrunoLintner2} in the case of an open-arc in
two dimensions, the combination $\mathbf{N}_\omega \mathbf{S}_\omega$ gives
rise to second-kind integral equations. In particular, the numerical results presented in~\cite{BrunoLintner2}
show that equations~\eqref{NSGoodD} and ~\eqref{NSGoodN} require
significantly smaller numbers of GMRES iterations than
equations~\eqref{SGood} and~\eqref{NGood} to achieve a given residual
tolerance. As demonstrated in this paper through a variety of numerical examples,
a similar reduction in iteration numbers results for three-dimensional problems as well.

\section{Outline of the proposed Nystr\"om solver\label{numer_framework}}
\subsection{Basic algorithmic structure\label{nystrom}}
In order to obtain numerical solutions of the surface integral
equations~\eqref{SGood}--\eqref{NSGoodN} we introduce an open-surface
version of the closed-surface Nystr\"om solver put forth
in~\cite{BrunoKunyansky}. This algorithm relies on
\begin{enumerate} 
\item A discrete set of nodes $\mathcal{N}= \{\mathbf{r}_i,i=1,\dots,N\}$ on
the surface $\Gamma$, which are used for both \emph{integration} and
\emph{collocation};
\item \label{pt_2} High-order integration rules which, using a given
discrete set of accurate approximate values $(\varphi_i)$ (resp. $(\psi_i)$)
of a smooth surface density $\varphi$, $\varphi_i\sim\varphi(\mathbf{r}_i)$
(resp. $\psi$, $\psi_i\sim\psi(\mathbf{r}_i)$), produce accurate
approximations of the quantities $\mathbf{S}_\omega[\varphi](\mathbf{r}_i)$
(resp. $\mathbf{N}_\omega[\psi](\mathbf{r}_i)$), see Sections~\ref{sec_POU}
through~\ref{efficiency}; and
\item The iterative linear algebra solver GMRES~\cite{saad}, for solution of
the discrete versions of equations~\eqref{SGood}--\eqref{NSGoodN} induced by
the approximations mentioned in point~\ref{pt_2}.
\end{enumerate}
The fact that the same set of Nystr\"om nodes is used for
integration and collocation facilitates evaluation of the composite operator
$\mathbf{N}_\omega \mathbf{S}_\omega$ through simple subsequent application
of the discrete versions of the operators $\mathbf{S}_\omega$ and
$\mathbf{N}_\omega$. 

\begin{figure}
\center
\includegraphics[scale=0.35]{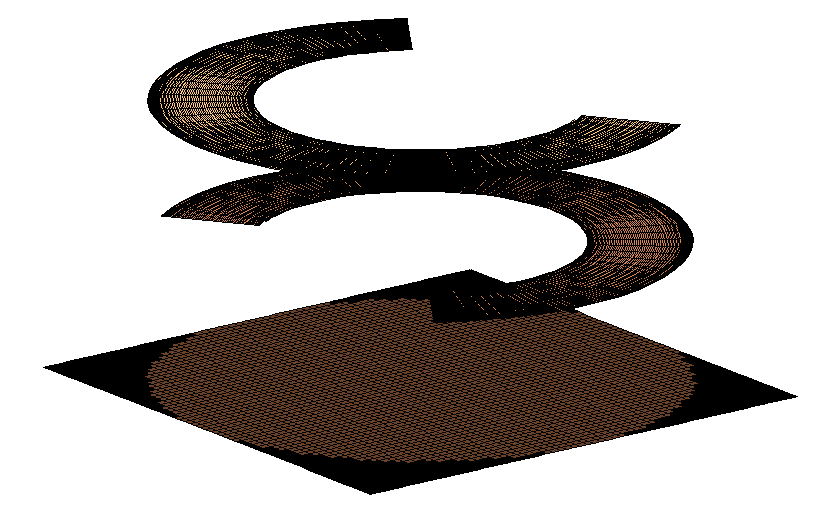}
\includegraphics[scale=0.35]{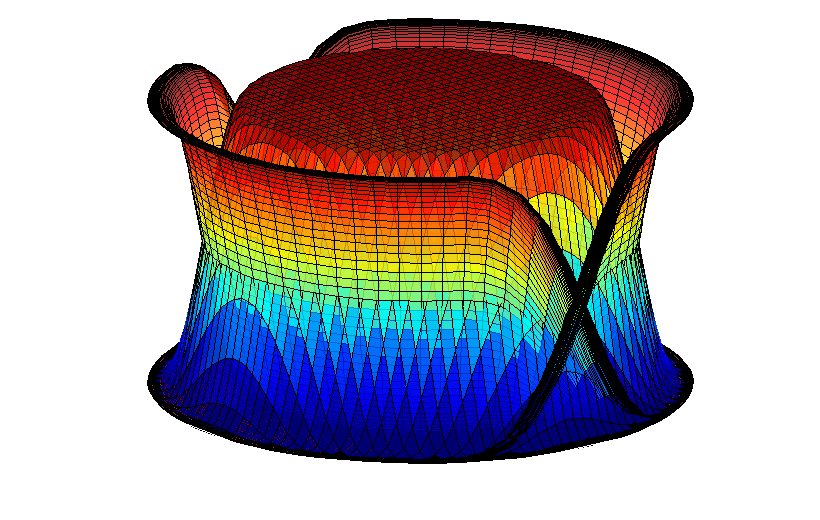}
\caption{Patches, partition of unity and discretization for a disc. Left:
the disc is covered by an interior patch and two edge patches. Right:
partition of unity functions $W^q_i$ supported on the patches. Notice the
quadratic refinement along the edges.  }
  \label{discPOU}
\end{figure}

\subsection{Partition of unity and Nystr\"om nodes}\label{sec_POU}
The integration rules mentioned in point~\ref{pt_2} in Section~\ref{nystrom}
rely on a decomposition of a given open surface (screen) $\Gamma$ as a union
\begin{equation}\label{patch_union}
\left (\bigcup_{q=1}^{Q_1} \mathcal{P}_1^{q}\right)
\bigcup \left (\bigcup_{q=1}^{Q_2} \mathcal{P}_2^{q}\right)
\end{equation}
of overlapping patches, including interior patches $\mathcal{P}_1^{q}$,
$q=1,\dots, Q_1$, and edge patches $\mathcal{P}_2^{q}$, $q=1,\dots,
Q_2$. For each $q$, the interior patch $\mathcal{P}_1^q$ (resp. edge patch
$\mathcal{P}_2^q$) is assumed to be parametrized by an invertible smooth
mapping $\mathbf{r}_1^q = \mathbf{r}_1^q(u,v)$, $\mathbf{r}_1^q: \quad
\mathcal{H}_1^q \rightarrow \mathcal{P}_1^q$ (resp. $\mathbf{r}_2^q=
\mathbf{r}_2^q(u,v)$, $\mathbf{r}_2^q: \quad \mathcal{H}_2^q \rightarrow
\mathcal{P}_2^q$) defined over an open domain $\mathcal{H}_1^q\subset
\mathbb{R}^2$ (resp. $\mathcal{H}_2^q\subset \mathbb{R}^2\bigcap \{v \geq
0\}$). Note that, for the $q$-th edge patch, the restriction of the mapping
$\mathbf{r}_2^q$ to the set $\mathcal{H}_2^q\cap \{v = 0\}$ (which we assume
is non-empty for $q=1,\dots, Q_2$) provides a parametrization of a portion
of the edge of $\Gamma$; see Figure~\eqref{discPOU}.
Following~\cite{BrunoKunyansky}, further, we introduce a partition of unity
(POU) subordinated to the set of overlapping patches mentioned above. In
detail, the POU we use is a set of non-negative functions $W_1^{q}$ and
$W_2^{q}$ defined on $\Gamma$, $q=1,\dots,Q_1$ and $q=1,\dots,Q_2$, such
that, for all $q$, $W_1^{q}$ vanishes outside $\mathcal{P}_1^q$, $W_2^{q}$
vanishes outside $\mathcal{P}_2^q$, and the relation
\[
\sum_{q=1}^{Q_1} W_1^{q} +\sum_{q=1}^{Q_2} W_2^{q} = 1
\]
holds throughout $\Gamma$. 
The POU can be used to decompose the integral of any function over the surface into a patch-wise sum of the form
\begin{equation}\label{dec}
\begin{split}
\int_\Gamma f(\mathbf{r}')dS' =  \sum\limits_{j=1}^2\sum \limits_{q=1}^{Q_j}  \int_{\mathcal{H}_j^q} f\left(\mathbf{r}^q_j(u,v)\right)W_j^q\left(\mathbf{r}^q_j(u,v)\right) J^q_j(u,v)dudv,
\end{split}
\end{equation}
where $J^q_j(u,v)$ denote the Jacobian associated with the
parametrization $\mathbf{r}_j^q$.
 At this stage we define the set of Nystr\"om nodes: introducing, for each
$q$, a tensor-product mesh $\{(u^{q,j}_\ell,v^{q,j}_m)\}$ within
$\mathcal{H}_j^q$ (for $j=1$, $2$), we obtain points $\mathbf{r}^{q,j}_{l,m}
= \mathbf{r}_j^q(u^{q,j}_\ell,v^{q,j}_m)$ on the surface $\Gamma$. For every
$j=1,2$ and every $q=1,\dots ,Q_j$, the set of nodes
$\mathbf{r}^{q,j}_{l,m}$ for which the POU function $W_j^q$ associated with
the patch $\mathcal{P}_j^q$ does not vanish
\begin{equation}
\mathcal{N}^{q,j} = \left \{ \mathbf{r}^{q,j}_{l,m} \, :\,
W^q_j(\mathbf{r}^{q,j}_{l,m})>0\right \}
\end{equation}
defines the set of \emph{Nystr\"om nodes} on the the patch
$\mathcal{P}_j^q$. The overall set $\mathcal{N}$ of Nystr\"om nodes on the
surface $\Gamma$ mentioned in Section~\ref{nystrom}, is given by
\begin{equation}
\mathcal{N} = \bigcup_{j=1}^{2} \bigcup_{q=1}^{Q_j} \mathcal{N}^{q,j}.
\end{equation}

\begin{remark}\label{remark_classes}
For later reference we introduce the classes of functions
\begin{equation}\label{class1}
\mathcal{D}_1^q=\left \{\phi_1^q\in C^\infty\left(\mathcal{H}^q_1\right)\; : \;
\mathrm{supp}(\phi_1^q)\Subset\mathcal{H}^q_1 \right\}\; ,\; q=1,\dots,Q_1\;
,
\end{equation}
\begin{equation}\label{class2}
\mathcal{D}_2^q=\left \{\phi_2^q\in C^\infty\left(\mathcal{H}^q_2\right)\; :
\; \mathrm{supp}(\phi_2^q)\Subset\mathcal{H}^q_2 \right\}\; ,\;
q=1,\dots,Q_2\; ,
\end{equation}
where $C^\infty\left(\mathcal{H}^q_1\right)$ denotes the set of infinitely
differentiable functions defined on the open set $\mathcal{H}^q_1$,
$C^\infty\left(\mathcal{H}^q_2\right)$ denotes the set of functions defined
on the set $\mathcal{H}^q_2$ that are infinitely smooth on $\mathcal{H}^q_2$
up to and including the edge $\mathcal{H}^q_2\cap \{ v=0\}$, and where, for
sets $A$ and $B\subseteq \mathbb{R}^n$ the notation $A\Subset B$ indicates
that the closure of $A$ in $\mathbb{R}^n$ is a compact subset of
$\mathbb{R}^n$ that is contained in $B$.
\end{remark}

\subsection{Canonical decomposition and high-order quadrature rules\label{num_dec}}
The high-order numerical quadratures required by step~2. in
Section~\ref{nystrom} are obtained by applying the patch-wise
decomposition~\eqref{dec} to the weighted integral operators
$\mathbf{S}_\omega$ and $\mathbf{N}_\omega$ and reducing each one of them to
a sum of patch-integrals which, as shown in Sections~\ref{sec_S}
through~\ref{sec_combined}, can be classified into six distinct
\textit{canonical types}. For each of these canonical types we construct, in
Sections~\ref{sec_interior} and~\ref{sec_exterior}, spectrally convergent
quadrature rules on the basis of the Nystr\"om points
$\mathcal{N}^{q,j}$. Suggestions concerning selection of sizes of the edge
and interior patches, which are dictated on the basis of efficiency and
accuracy considerations, are put forth in Section~\ref{param_sel}.

\subsection{Computational implementation and efficiency\label{efficiency}}
The high-order methods presented in the following sections enable accurate
and fast computation of each one of the six canonical integral types
mentioned in Section~\ref{num_dec}. For added efficiency, however, our
solver exploits common elements that exist between the various canonical
integration algorithms, to avoid re-computation of quantities such as the
trigonometric functions associated with the Green's function, partition of
unity functions, integral weights, etc. Additional efficiency could be
gained by incorporating an acceleration method (see
e.g.~\cite{BrunoKunyansky} and references therein) and code parallelization.

\begin{figure}[h]
\center
\includegraphics[scale=0.4]{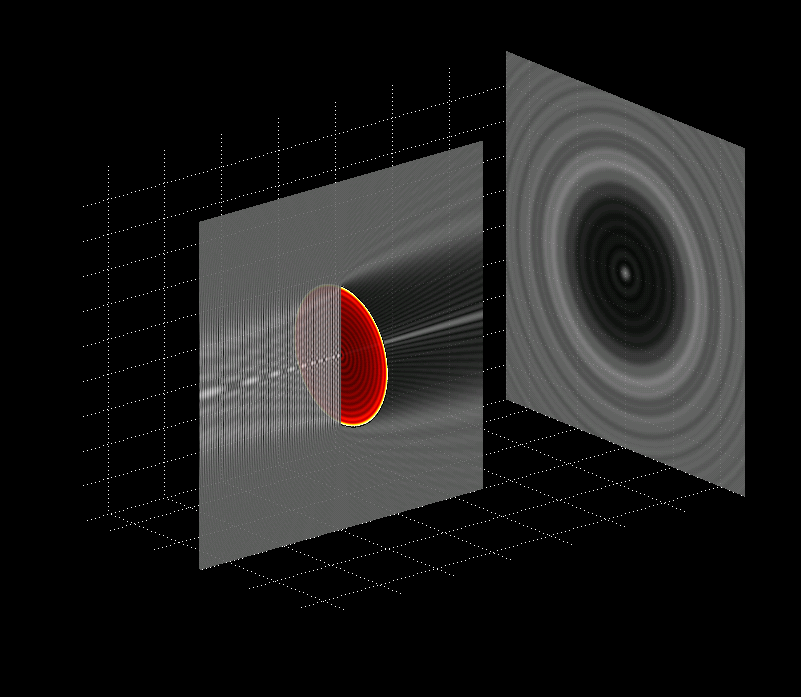}
\caption{Diffraction by an infinitely thin disc: solution to the Neumann
  problem for a disc of diameter $24\lambda$ under normal incidence. The
  famous Poisson spot is clearly visible at the center of the shadow area;
  see also Figure~\ref{poisson_fig}. The coloring on the disc represents
  the values of the surface unknown $\psi$.
  \label{disc_fig}} 
\end{figure}

\begin{figure}[h]
\center
\includegraphics[scale=0.4]{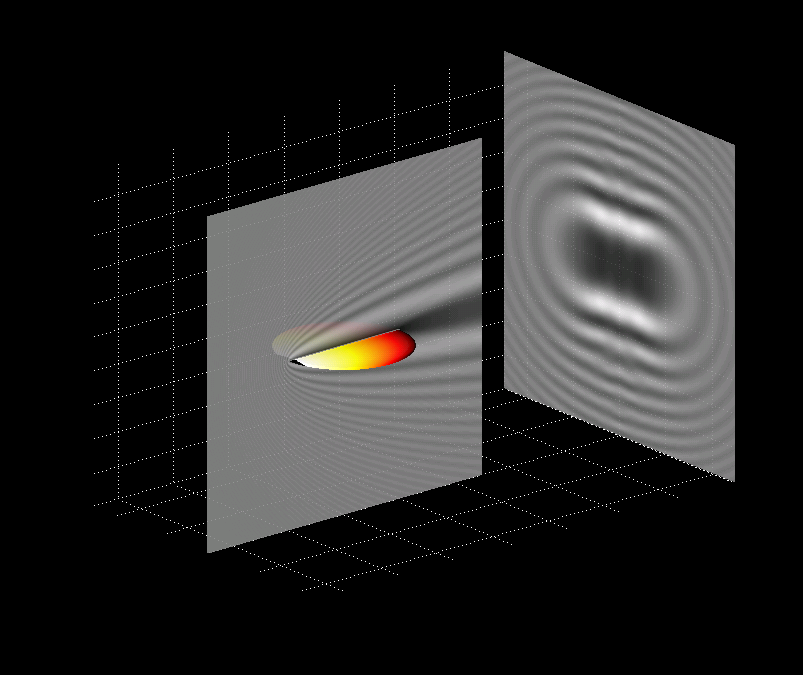}
\caption{Diffraction by an infinitely thin disc: solution to the Dirichlet
  problem for a disc of diameter $24\lambda$ under incidence parallel to the
  disc. The coloring on the disc represents the values of the surface
  unknown $\phi$.
  \label{disc_fig2}}
\end{figure}

\section{Canonical singular-integral decomposition of the operator $\mathbf{S}_\omega$\label{sec_S}}
In view of equation~\eqref{dec}, we express the weighted single-layer
operator $\mathbf{S}_\omega$,
\begin{equation}\label{single_expl}
\mathbf{S}_\omega[\varphi](\mathbf{r})=\int_\Gamma G_k(\mathbf{r},\mathbf{r}')\frac{\varphi(\mathbf{r}')}{\omega(\mathbf{r}')}dS',
\end{equation}
in the form
\begin{equation}\label{Spatch}
\mathbf{S}_\omega=  \sum \limits_{q=1}^{Q_1}  \mathcal{S}_1^q+\sum \limits_{q=1}^{Q_2}  \mathcal{S}_2^q,
\end{equation}
where
\begin{equation}\label{Sq}
\mathcal{S}^q_j[\varphi](\mathbf{r})=\int_{\mathcal{H}_j^q}
G_k(\mathbf{r}_j^q(u,v),\mathbf{r})\frac{\varphi\left(\mathbf{r}^q_j(u,v)\right)}{\omega\left(\mathbf{r}^q_j(u,v)\right)}
J^q_j(u,v)W^q_j(\mathbf{r}^q_j(u,v))dudv, \quad j=1,2.
\end{equation}
In Sections~\ref{int_patch} and~\ref{edge_patch} the integrals~\eqref{Sq}
are expressed in terms of canonical integrals of various types.
\subsection{Interior patch decomposition\label{int_patch}}
 For an interior patch $\mathcal{P}^q_1$ and for a point
 $\mathbf{r}\in\Gamma\setminus \mathcal{P}^q_1$, the integrand in~\eqref{Sq}
 is smooth and compactly supported within the domain of integration
 $\mathcal{H}^q_1$---since the weight $\omega(\mathbf{r})$ is smooth and
 nonzero away from the edge, and since the POU function $W^q_1$ vanishes
 outside $\mathcal{H}^q_1$---and, thus, the integral~\eqref{Sq} gives rise
 to our first canonical integral type:
 
\begin{equation}\label{Iqreg}
\begin{split}
&\mbox{Canonical Integral of Type I}\\
\mathcal{I}^{q,reg}_1[\phi_1]=\int_{\mathcal{H}_1^q}\phi_1&(u,v) du dv\quad ,
\quad \phi_1 \in \mathcal{D}_1^q\; ,
\end{split}
\end{equation}
see Remark~\ref{remark_classes}. For a point $\mathbf{r}\in\mathcal{P}^q_1$,
$\mathbf{r}=\mathbf{r}^q_1(u_0,v_0)$ for some $(u_0,v_0) \in
\mathcal{H}^q_1$, on the other hand, the integrand of~\eqref{Sq} with $j=1$
has an integrable singularity at the point $(u_0,v_0)$
(cf. equation~\eqref{Green}). Following~\cite{BrunoKunyansky} we express the
kernel as a sum of a localized singular part, and a smooth remainder,
$G_k=G_k^{sing} +G_k^{reg}$, where
\begin{equation}\label{Phi_sep}
G_k^{sing}=\eta_\mathbf{r}G_k^{re} ,\quad G_k^{reg}= (1-\eta_\mathbf{r})G_k^{re} + i G_k^{im} .
\end{equation}
Here, $G_k^\textit{re}(\mathbf{r},\mathbf{r}')=\frac{\cos( k
  |\mathbf{r}-\mathbf{r}' |)}{ |\mathbf{r}-\mathbf{r}' |}$ and
$G_k^{im}(\mathbf{r},\mathbf{r}')=\frac{\sin( k |\mathbf{r}-\mathbf{r}' |)}{
  |\mathbf{r}-\mathbf{r}' |}$ denote the real and imaginary parts of the
kernel $G_k(\mathbf{r},\mathbf{r}')$, respectively, and $\eta_\mathbf{r}$ is
a smooth function which vanishes outside a neighborhood of the point
$\mathbf{r}$. As in the previous reference, the collection of all pairs
$(\eta_\mathbf{r}, 1-\eta_\mathbf{r})$ for $\mathbf{r}\in\Gamma$ is called a
\textit{floating partition of unity}.  The integral that arises as $G_k$ is
replaced in~\eqref{Sq} by $G_k^{reg}$, has a smooth integrand which is
compactly supported within $\mathcal{H}_1^q$; clearly, this is an integral
of canonical type I. The integral obtained by substituting $G_k$ by
$G^{sing}_k$, on the other hand, gives rise to our second canonical type \\
\begin{equation}\label{Iqsing}
\begin{split}
&\mbox{Canonical Integral of Type II}\\
\mathcal{I}^{q,sing}_1[\phi_1](u_0,v_0)=\int_{\mathcal{H}_1^q}&\frac{\phi_1(u,v)}{|
  \mathbf{R}|} du dv,\quad \phi_1 \in \mathcal{D}_1^q
\end{split}
\end{equation}
where for the sake of conciseness, we have set $\mathbf{R}=\mathbf{r}^q_1(u_0,v_0)-\mathbf{r}_1^q(u,v)$ and where the point $(u_0,v_0)$ belongs to $\mathcal{H}^q_1$.

\subsection{Edge-patch decomposition~\label{edge_patch}}
The edge singularity on an edge patch $\mathcal{P}^q_2$ is characterized in
terms of the asymptotic form~\eqref{asymp}. In what follows we assume, as we
may, that on each edge patch,the weight $\omega$ is given by an expression
of the form
  \begin{equation}
\omega(\mathbf{r}^q_2(u,v))=\omega_2^q(u,v)\sqrt{v}, 
  \end{equation}
where the function $\omega_2^q(u,v)$ is \textit{smooth up to the edge and it
does not vanish anywhere along the edge}.  It follows that for an edge patch
$\mathcal{P}^q_2$ and for a point $\mathbf{r}\in\Gamma\setminus
\mathcal{P}^q_2$, the operator $\mathcal{S}_2^q$ defined in
equation~\eqref{Sq} for $j=2$ takes the form of an integral of our third
canonical type:

\begin{equation}\label{Iqreg_w}
\begin{split}
&\mbox{Canonical Integral of Type III}\\
\mathcal{I}^{q,reg}_2[\phi_2](u_0,v_0)=\int_{\mathcal{H}_2^q}&\phi_2(u,v)\frac{du
  dv}{\sqrt{v}}, \quad \phi_2 \in \mathcal{D}^q_2.
\end{split}
\end{equation}
Finally, for an edge patch $\mathcal{P}^q_2$ and for a point
$\mathbf{r}\in\mathcal{P}^q_2$ we once again use the floating partition of
unity to decompose the Green function as a sum of a singular and a regular
term. The regular term results in a canonical integral of Type~III, and the
singular term gives rise to our fourth canonical type:

\begin{equation}\label{Iqsing_w}
\begin{split}
&\mbox{Canonical Integral of Type IV}\\
\mathcal{I}^{q,sing}_2[\phi_2]=\int_{\mathcal{H}_2^q}&\frac{\phi_2(u,v)}{|\mathbf{R}|}
\frac{du dv}{\sqrt{v}}, \quad \phi_2 \in \mathcal{D}_2^q
\end{split}
\end{equation}
where $\mathbf{R}=\mathbf{r}^q_2(u_0,v_0)-\mathbf{r}^q_2(u,v)$.

 \section{Canonical decomposition of the operator $\mathbf{N}_\omega$\label{sec_N}}
In view of equations~\eqref{Ndef} and~\eqref{Somega} the operator
$\mathbf{N}_\omega$ is given by the point-wise limit
 \begin{equation}\label{Nomegadef}
 \mathbf{N}_\omega[\psi](\mathbf{r})\equiv \lim\limits_{z\rightarrow 0}
 \frac{\partial }{\partial \textbf{n}_{\mathbf{r}}}\int_\Gamma
 \frac{\partial
 G_k(\mathbf{r},\mathbf{r}'+z\textbf{n}_{\mathbf{r}'})}{\partial
 \textbf{n}_\mathbf{r}'}\psi(\mathbf{r}')
 \omega(\mathbf{r}')dS',\textbf{ }\quad \mathbf{r}\in \Gamma.
 \end{equation}
Following the open-arc derivation~\cite{LintnerBruno,Monch,ColtonKress1} we
obtain an adequate expression for this ``hypersingular operator'' by taking
advantage of the following lemma.

\begin{lemma}\label{Ndec}
The operator $\mathbf{N}_\omega$ can be expressed in the form
\begin{equation}\label{Ndec_eq}
\mathbf{N}_\omega=\mathbf{N}_\omega^g
+\mathbf{N}_\omega^{pv}\mathcal{T}_\omega
\end{equation}
where, denoting the surface gradient with respect to $\mathbf{r}'$ by
$\nabla^s_{\mathbf{r}'}$ and letting $[\,\cdot\, ,\cdot\,]$ denote the
vector product, the operators $\mathbf{N}_\omega^g$,
$\mathbf{N}_\omega^{pv}$ and $\mathcal{T}_\omega$ are given by
\begin{equation}\label{Ns}
\mathbf{N}_\omega^g[\psi](\mathbf{r})=k^2\int_\Gamma
G_k(\mathbf{r},\mathbf{r}')\psi(\mathbf{r}')\omega(\mathbf{r}')\textbf{n}_{\mathbf{r}'}.\textbf{n}_{\mathbf{r}}dS',\end{equation}
\begin{equation}\label{Tomega}
\mathcal{T}_\omega[\psi](\mathbf{r}')=\omega^2(\mathbf{r}')
\nabla^s_{\mathbf{r}'}[\psi](\mathbf{r}')+\frac{\psi(\mathbf{r}')}{2}\nabla^s_{\mathbf{r}'}
[\omega^2](\mathbf{r}'),\quad \mbox{and}
\end{equation}
\begin{equation}\label{Npv}
\mathbf{N}_\omega^{pv}[\mathbf{T}](\mathbf{r})=
p.v. \int_\Gamma\big[\nabla_{\mathbf{r}}G_k(\mathbf{r},\mathbf{r}'),\left[\textbf{n}_{\mathbf{r}'},\mathbf{T}(\mathbf{r}')\right]\big]\cdot\textbf{n}_{\mathbf{r}}\frac{dS'}{\omega(\mathbf{r}')}.
\end{equation}
\end{lemma}
\begin{proof}
See Appendix~\ref{Appendix_N}.
\end{proof}
As shown in Sections~\ref{Ns_dec_sec} through~\ref{DT}, we may evaluate
$\mathbf{N}_\omega$ by relying on Lemma~\ref{Ndec} and using quadratures for
various types of canonical integrals.

\subsection{Canonical decomposition of the operator $\mathbf{N}^g_\omega$ (equation~\eqref{Ns})\label{Ns_dec_sec}}
Calling $\psi_2=\psi \omega^2$ we re-express~\eqref{Ns} in the form
\begin{equation}\label{Ns_dec}
\mathbf{N}_\omega^g[\psi](\mathbf{r})=k^2\int_\Gamma G_k(\mathbf{r},\mathbf{r}')\frac{\left(\psi_2(\mathbf{r}')\textbf{n}_{\mathbf{r}'}.\textbf{n}_{\mathbf{r}}\right)}{\omega(\mathbf{r}')}dS'.
\end{equation}
Since $\omega^2(\mathbf{r})$ is a smooth function of $\mathbf{r}$ throughout
$\Gamma$, a construction similar to the one used for~\eqref{single_expl}
yields a decomposition of the operator $\mathbf{N}_\omega^s$ in terms of
canonical integrals of types I-IV; see Section~\ref{sec_S}.

\subsection{Canonical decomposition of the operator $\mathcal{T}_\omega$ (equation~\eqref{Tomega})}
Making use once again of the POU introduced in Section~\ref{sec_POU}, we
obtain the decomposition
$$\nabla^s_\mathbf{r}[\psi](\mathbf{r}) =\sum_{q=1}^{Q_1}
\nabla^s_\mathbf{r}[\psi W_1^q](\mathbf{r}) + \sum_{q=1}^{Q_2}
\nabla^s_{\mathbf{r}}[\psi W_2^q](\mathbf{r}),$$ of the surface gradient.
The evaluation of the $q$-th term in each one of these sums requires the
calculation of partial derivatives of the form
\begin{equation}\label{partial1}
\frac{\partial \phi_1(u,v)}{\partial u}\quad ,\quad
\frac{\partial\phi_1(u,v)}{\partial v}\; ,
\end{equation}
\begin{equation}\label{partial2}
\frac{\partial \phi_2(u,v)}{\partial u}\quad\mbox{and} \quad
\frac{\partial\phi_2(u,v)}{\partial v}
\end{equation}
for functions $\phi_1\in\mathcal{D}_1^q$ and
$\phi_2\in\mathcal{D}_2^q$. These partial derivatives can be evaluated
efficiently and with high-order accuracy by means of the differentiation
methods introduced in Sections~\ref{sec_interior} and~\ref{sec_exterior}
below.  In view of equation~\eqref{Tomega}, use of such high-order rules
enables high-order evaluation of the operator
$\mathcal{T}_\omega[\psi](\mathbf{r})$.

\subsection{Canonical decomposition of the operator $\mathbf{N}^{pv}_\omega$ (equation~\eqref{Npv})\label{DT}}
It is easy to check that, for a smooth field $\mathbf{T}$, $\mathbf{N}^{pv}_\omega[\mathbf{T}]$ can be evaluated
as a linear combination of functions of the form
$\mathbf{D}^{\mathbf{V}_\ell}_\omega[\phi_\ell]$, where $\mathbf{V}_\ell$ is a vector
quantity that varies with $\mathbf{r}$ but is independent of $\mathbf{r}'$,
where the operator $\mathbf{D}^{\mathbf{V}}_\omega$ is defined by
\begin{equation}\label{Domega}
\mathbf{D}^{\mathbf{V}}_\omega[\psi](\mathbf{r})=\mbox{p.v.} \int_\Gamma \left\{
\nabla_{\mathbf{r}}
G_k(\mathbf{r},\mathbf{r}')\cdot\mathbf{V}\right\}\frac{\psi(\mathbf{r}')}{\omega(\mathbf{r}')}dS',
\end{equation}
and where $\phi_\ell$ are smooth functions.  Applying the
decomposition~\eqref{dec} to the operator defined in equation~\eqref{Domega}
yields\begin{equation} \mathbf{D}^{\mathbf{V}}_\omega[\psi]=\sum\limits_{q=1}^{Q_1}
\mathbf{D}_1^{\mathbf{V},q}[\psi](\mathbf{r}) + \sum\limits_{q=1}^{Q_2}
\mathbf{D}_2^{\mathbf{V},q}[\psi](\mathbf{r}),
\end{equation}
 where 
 \begin{equation}\label{Djq}
 \mathbf{D}_j^{\mathbf{V},q}[\psi](\mathbf{r})=\mbox{p.v.}\int_{\mathcal{H}^q_j}\left\{
 \nabla_{\mathbf{r}}G_k(\mathbf{r},\mathbf{r}^q_j(u,v)) \cdot \mathbf{V} \right\} \frac{\psi(\mathbf{r}^q_j(u,v)) }{\omega(\mathbf{r}^q_j(u,v))}W_j^q(\mathbf{r}^q_j(u,v)) J^q_j(u,v)dudv.
 \end{equation}
We can
express the operator $\mathbf{D}_j^{\mathbf{V},q}$ as the sum of a hypersingular
operator and a weakly singular operator whose respective kernels are defined
by the split
\begin{equation}\label{grad_split}
\nabla_{\mathbf{r}}G_k(\mathbf{r},\mathbf{r}')\cdot\mathbf{V}=G_k^{pv} +
G_k^{ws} , \quad{}
G_k^{pv}=-\frac{(\mathbf{r}-\mathbf{r}')\cdot\mathbf{V}}{|\mathbf{r}-\mathbf{r}'|^3},
\end{equation}
where the residual kernel $G_k^\textit{ws}$ equals a sum of functions which are
either smooth or weakly singular with singularity $\frac{1}{|\mathbf{r}-\mathbf{r}'|}$. Using the partition-of-unity split embodied
in equation~\eqref{dec}, the operator with kernel $G_k^\textit{ws}$ can be
expressed in terms of integrals of canonical types I-IV.  The hypersingular
operator with kernel $G^{pv}_k$ on the other hand gives rise to our fifth
and sixth canonical types:
\begin{equation}\label{Ipv1}
\begin{split}
&\mbox{Canonical Integral of Type V}\\
\mathcal{I}_1^{q,pv}[\phi_1](u_0,v_0)=\mbox{p.v.}\int_{\mathcal{H}_1^q}&\frac{\mathbf{R}\cdot\mathbf{V}}{|\mathbf{R}|^3}\phi_1(u,v)dudv.
\end{split}
\end{equation}

\begin{equation}\label{Ipv2}
\begin{split}
&\mbox{Canonical Integral of Type VI}\\
\mathcal{I}_2^{q,pv}[\phi_2](u_0,v_0)=\mbox{p.v.}\int_{\mathcal{H}_2^q} 
&\frac{\mathbf{R}\cdot\mathbf{V}}{|\mathbf{R}|^3}\phi_2(u,v)\frac{dudv}{\sqrt{v}}.
\end{split}
\end{equation}

\section{Canonical decomposition of the composite operator $\mathbf{N}_\omega \mathbf{S}_\omega$}\label{sec_combined}
While the action of the composite operator $\mathbf{N}_\omega
\mathbf{S}_\omega$ on a function $\phi$ can be evaluated by producing first
$\psi = \mathbf{S}_\omega[\phi]$ and then evaluating
$\mathbf{N}_\omega[\psi]$, both of which can be obtained by the methods
described in the previous sections, we have found it advantageous in
practice to proceed differently, on the basis of the
expression~\eqref{Ndec_eq}; see Remark~\ref{Remark_sqrt} for more
details. Using the decomposition~\eqref{Ndec_eq}, we first evaluate the term
$\mathbf{N}^{g}_\omega \mathbf{S}_\omega[\phi]$ by means of a direct
composition: we compute $\mathbf{S}_\omega[\phi]$ and then apply
$\mathbf{N}^{g}_\omega$ to the result using the decompositions put forth in
Sections~\ref{sec_S} and~\ref{Ns_dec_sec} respectively.  To evaluate the
second term $\mathbf{N}^{pv}_\omega \mathbf{S}_\omega[\phi]$, on the other
hand, we first evaluate the quantity $\mathcal{T}_\omega
\mathbf{S}_\omega[\phi]$ by expressing the surface gradient of
$\mathbf{S}_\omega[\phi]$ required by equation~\eqref{Tomega} as
\begin{equation}\label{eq_grad_dec}
\nabla^s_\mathbf{r} S_\omega[\phi](\mathbf{r})=p.v.\int_\Gamma
\nabla^s_\mathbf{r}
G_k(\mathbf{r},\mathbf{r}')\frac{\phi(\mathbf{r}')}{\omega(\mathbf{r}')}dS'=\sum_{\ell=1,2}\mathbf{D}_\omega^{\tau_\ell(\mathbf{r})}[\phi](\mathbf{r}),
\end{equation}
where $\tau_1(\mathbf{r})$ and $\tau_2(\mathbf{r})$ denote two orthogonal
tangent vectors to $\Gamma$ at the point $\mathbf{r}$ which vary smoothly
with $\mathbf{r}$; the corresponding surface gradient of $\omega^2$ can be
obtained by direct differentiation of a closed form expression, if
available, or by means of the differentiation methods put forth in this
paper. The terms in the sum on the right-hand side of~\eqref{eq_grad_dec}
are of the form given in equation~\eqref{Domega} with
$\mathbf{V}=\tau_\ell(\mathbf{r})$ ($\ell =1,2$), and thus can be expressed
in terms of the canonical integrals of type I-VI, as outlined in
Section~\ref{DT}.

\begin{figure}[h]
\center
\includegraphics[scale=0.5]{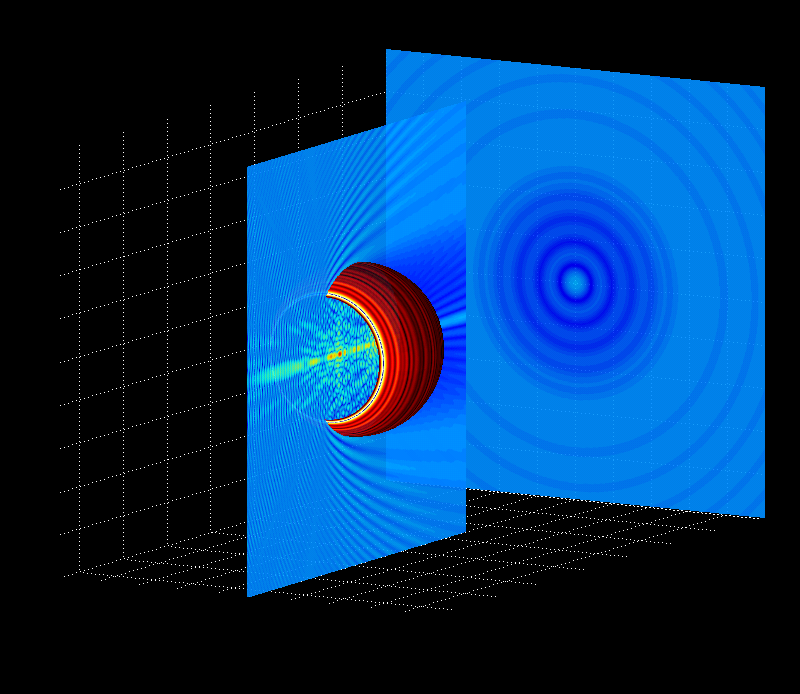}
\caption{Neumann problem on a spherical cavity of diameter $18\lambda$. The
coloring on the spherical wall represents the values of the surface unknown
$\psi$.\label{cavity_fig}}.
\end{figure}

\begin{figure}[h]
\center
\includegraphics[scale=0.5]{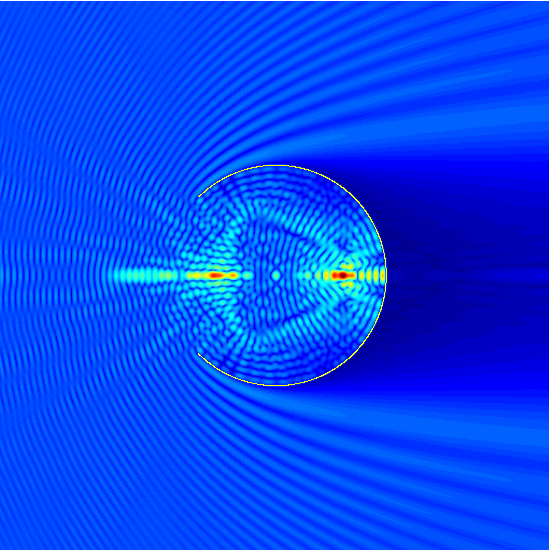}
\caption{Total Field inside a spherical cavity of diameter $18\lambda$,
Dirichlet Problem. \label{cavity_fig2}}.
\end{figure}

\section{High-order evaluation of interior-patch operators}\label{sec_interior}
In this section we describe our algorithms for evaluation of the
interior-patch operators introduced in the previous sections, namely the
integral operators of type I, II and~V and the differentiation
operator~\eqref{partial1}. To do this, and in accordance with
Section~\ref{sec_POU}, we assume that the nodes $(u^{q,1}_{\ell},v^{q,1}_{m})$
($u^{q,1}_\ell=u^{q,1}_0 + \ell h^{q,1}_u$, $v^{q,1}_m=v^{q,1}_0 + m
h^{q,1}_v$, $\ell=1,\dots L^1_{q}$, $m=1\dots M^1_{q}$), discretize a
rectangle that contains $\mathcal{H}^q_1$. 

\subsection{Type I Integral (Regular)}
We evaluate the canonical Type I integral defined in equation~\eqref{Iqreg}
by means of a simple trapezoidal sum over the grid: as noted
in~\cite{BrunoKunyansky} the periodicity ($\phi_1$ is compactly supported)
and smoothness of the integrand gives rise to super-algebraic convergence in
this case.

\begin{figure}[h]
\center
\includegraphics[scale=0.5]{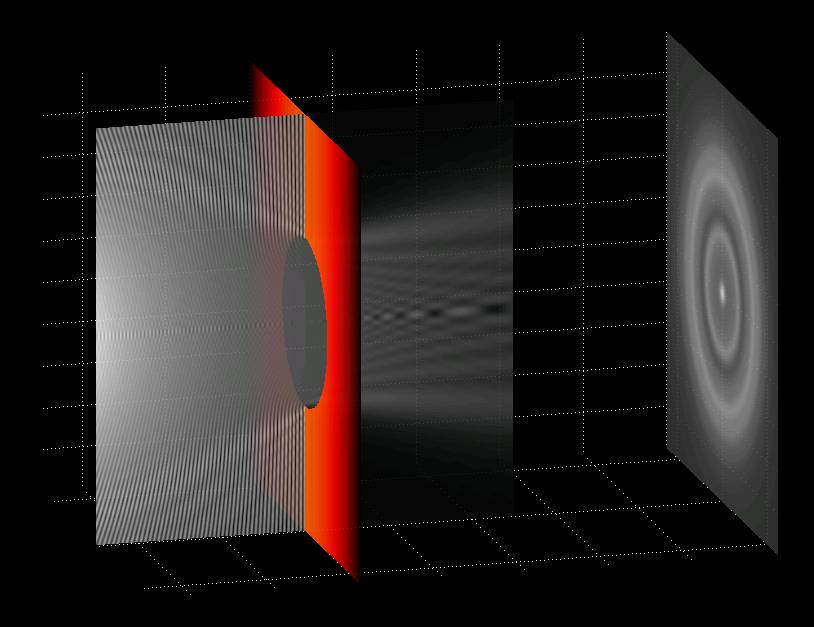}
\caption{Diffraction by a circular aperture: solution to the Neumann problem
  for an aperture of diameter $24\lambda$ under point source
  illumination. The source, which is not visible here, is located to the
  left of the displayed area. The coloring on the plane is introduced for
  visual quality, and it does not represent any physical quantity.
  \label{aperture_fig}}
\end{figure}

\begin{figure}[h]
\center
\includegraphics[scale=0.25]{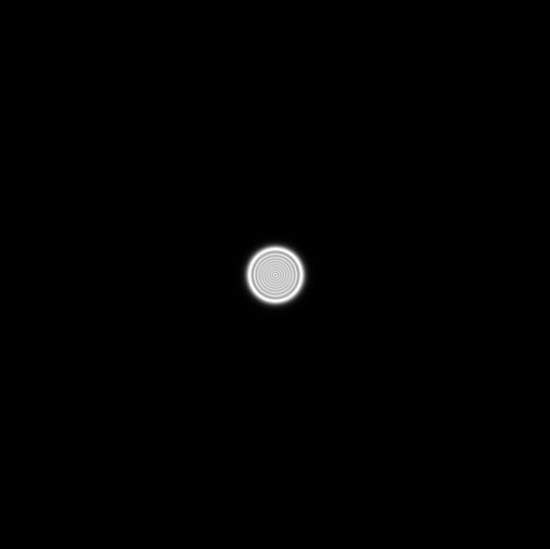}
\includegraphics[scale=0.25]{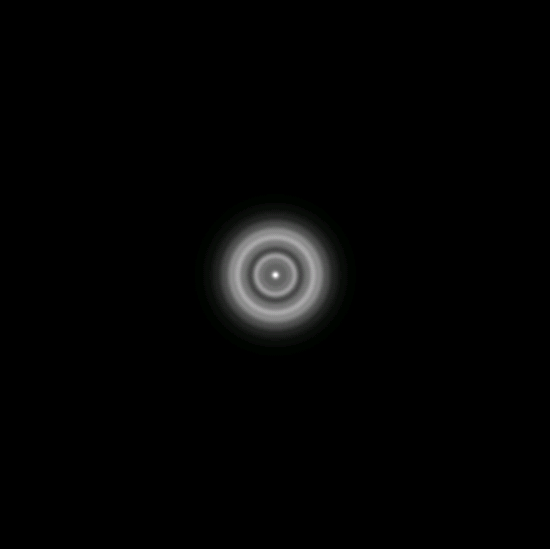}
\includegraphics[scale=0.25]{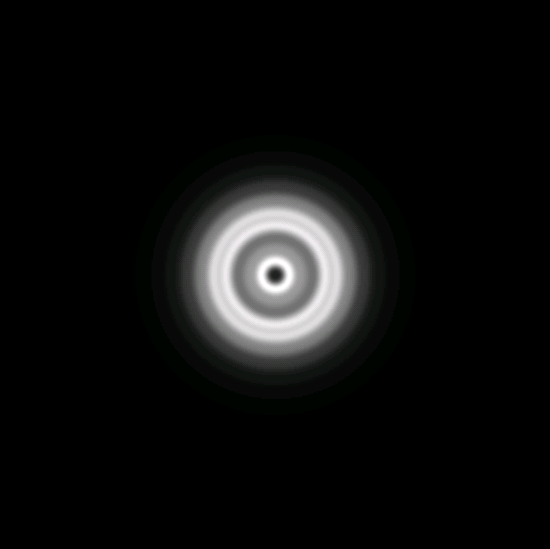}
\includegraphics[scale=0.25]{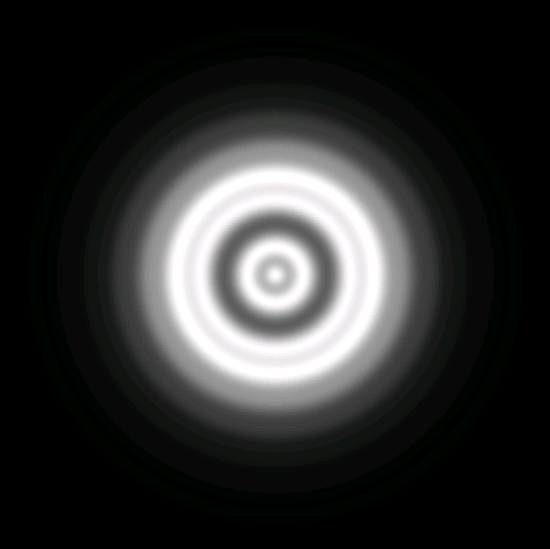}

\caption{Field diffracted by the circular aperture configuration depicted in
  Figure~\ref{aperture_fig}. From top left to bottom right, depiction of the
  diffracted field at observation screens located at distances of
  $6\lambda$, $60\lambda$, $120\lambda$ and $240\lambda$ behind the
  punctured plane. A dark-spot (the Poisson shadow) can be observed at the
  center of the illuminated area in the third-to-left
  image.\label{aperture_fringes}}.
\end{figure}

\subsection{Partial Derivatives}
In view of the smoothness and periodicity of the function $\phi_1$, a
standard two-dimensional FFT-based interpolation scheme based on the evenly
spaced grid values $\phi^{\ell,m}_1$ yields spectrally convergent
approximations of the function $\phi_1(u,v)$ and its derivatives; our
algorithm thus evaluates the derivatives required in
equation~\eqref{partial1} by performing a direct term by term
differentiation of the resulting Fourier representation.

\subsection{Type II  Integral (Singular)} \label{type_2_sec}
In order to resolve the singular integrand in equation~\eqref{Iqsing} we
utilize the polar change of variables introduced in~\cite{BrunoKunyansky}.
Defining $u(\rho,\theta)=u_0 +\rho \cos \theta$ and $v(\rho,\theta)=v_0 +
\rho\sin\theta$, we obtain
\begin{equation}\label{Itheta}
\mathcal{I}_1^{q,sing}[\phi_1](u_0,v_0)=\int_0^\pi
I^{q}_{\rho,1}[\phi_1](u_0,v_0,\theta)d\theta
\end{equation}
with
\begin{equation}\label{Iqrho}
I^q_{\rho,1}[\phi_1](u_0,v_0,\theta)=\int_{-\infty}^\infty \phi^\rho_1(\rho,\theta)
|\frac{|\rho| }{R} d\rho,
\end{equation}
where
\begin{equation}
R=|\mathbf{r}^q_1(u_0,v_0)-\mathbf{r}^q_1\left(u(\rho,\theta),v(\rho,\theta)\right) |,
\end{equation}
and where
\begin{equation}\label{phi1rho}
\phi^\rho_1(\rho,\theta)=\phi_1(u_0+\rho\cos\theta,v_0+\rho\sin\theta)
\end{equation} 
is a smooth function of $\rho$ and $\theta$ which vanishes for sufficiently
large values of $\rho$.  Since, as noted in~\cite{BrunoKunyansky}, the ratio
$\frac{|\rho|}{R}$ is a smooth function of $\rho$, the integral
$I^q_{\rho,1}[\phi_1](\theta, u_0,v_0)$ defined in~\eqref{Iqrho} can be computed
accurately via the trapezoidal rule with respect to $\rho$ for any value of
$\theta$. Similarly, applying the trapezoidal rule in the $\theta$ variable
gives rise to high-order convergence of the integral~\eqref{Itheta}, in view
of the $\pi$-periodicity of the integrand.

\begin{remark} Our application of the trapezoidal rule for  evaluation of
  $I_{\rho,1}^q(\theta,u_0,v_0)$ requires use of equidistant samples in the
  $\rho$ variable, which for most values of $\theta$, do not correspond to
  any of the original grid nodes $(u^{q,1}_\ell, v^{q,1}_m)$. To address
  this issue our solver relies on the FFT/cubic-spline interpolation
  technique presented in~\cite[Section 3]{BrunoKunyansky}, which allows for
  fast and efficient evaluation of the required equidistant $\rho$ samples.
\end{remark}

\begin{figure}[h]
\center
\includegraphics[scale=0.5]{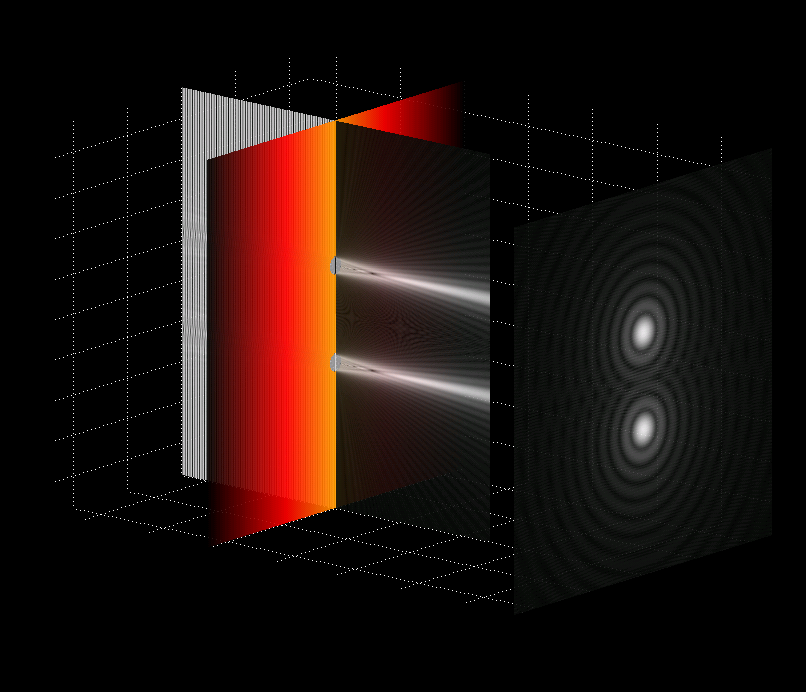}
\caption{Simulation of Young's experiment: diffraction by two circular
  apertures in a sound-hard plane (Neumann boundary conditions); the
  apertures are 24 wavelengths in diameter. The coloring on the plane, which
  is introduced for visual quality, does not represent any physical
  quantity. \label{Young_fig}}
\end{figure}

\subsection{Type V Integral (Principal Value)}\label{type_5_sec}
An application of the polar change of variables mentioned in
Section~\ref{type_2_sec} to the principal-value Type~V integral~\eqref{Ipv1}
results in the expression
 \begin{equation}\label{princ_value}
\mathcal{I}^{q,pv}_1[\phi_1](u_0,v_0)=\int_{0}^\pi I^{q,pv}_{\rho,1}[\phi_1](u_0,v_0,\theta)d\theta,
\end{equation}
where $I^{q,pv}_{\rho,1}[\phi_1](u_0,v_0,\theta)$ is given by the principal value integral
\begin{equation}\label{Ipv_rad}
I^{q,pv}_{\rho,1}[\phi_1](u_0,v_0,\theta)=p.v. \int_{-\infty}^\infty \frac{|\rho|\mathbf{R}\cdot\mathbf{V}}{R^3}\phi^\rho_1(\rho,\theta)d\rho.
\end{equation}
Here the function $\phi^\rho_1(\rho,\theta)$ is defined by
equation~\eqref{phi1rho} and we have set
$\mathbf{R}=\mathbf{r}^q_1(u_0,v_0)-\mathbf{r}^q_1\left(u(\rho,\theta) ,
v(\rho,\theta)\right)$. Equations~\eqref{princ_value} and~\eqref{Ipv_rad}
form the basis of our algorithm for evaluation of Type~V integrals.

Since both $\frac{|\rho|^3}{R^3}$ and
$\frac{\mathbf{R}\cdot\mathbf{V}}{\rho}$ are smooth functions of $\rho$ and
$\theta$, it is useful to consider the expression
\begin{equation}\label{I_pv_rho}
I^{q,pv}_{\rho,1}(u_0^q,v_0^q,\theta)=p.v. \int_{-\infty}^\infty \frac{|\rho|^3}{R^3}\frac{\mathbf{R}\cdot\mathbf{V}}{\rho}\frac{\phi^\rho_1(\rho,\theta)}{\rho}d\rho,
\end{equation}
for the integral~\eqref{Ipv_rad}. This is a 1-dimensional principal value
integral of the form
\begin{equation}\label{pv_1d}
I=\mbox{p.v.}\int_{-\infty}^\infty \frac{v(x)}{x}dx,
\end{equation}
where $v$ is a compactly supported smooth function.  Our algorithm proceeds
by evaluating this principal value integral by means of a trapezoidal rule
algorithm with integration nodes centered symmetrically around $x=0$---which,
as shown in~\cite{YingBirosZorin}, yields spectral accuracy for smooth and
periodic functions. In detail, letting $x_i=({i+\frac{1}{2}})/{M}$, after
appropriate scaling into the interval $[-1,1]$, our quadrature for the
integral~\eqref{pv_1d} is given by
\begin{equation}\label{pvcentral}
\mbox{p.v.}\int_{-1}^1 \frac{v(x)}{x}dx\sim \frac{1}{M}\sum_{i=-M}^{M-1}
\frac{v(x_i)}{x_i}.
\end{equation}
This expression provides spectrally accuracy as long as $v$ is a smooth
function of periodicity $2$. Our Type-V integration algorithm is completed
by trapezoidal integration in the $\theta$ variable to produce the
integral~\eqref{princ_value} with spectral accuracy.
\begin{remark} 
Application of the trapezoidal rule~\eqref{pvcentral} to compute the
integral~\eqref{I_pv_rho} requires evenly spaced samples in the $\rho$
variable, which, in addition, must also be \textit{symmetrically centered}
around $\rho=0$. To obtain such samples our algorithm proceeds in two steps:
1)~It uses the one-dimensional FFT/spline interpolation method presented
in~\cite[Section 3]{BrunoKunyansky} to produce evenly spaced samples of the
integrand in the $\rho$ variable, and 2)~It applies an FFT-based shift (see
Remark~\ref{shift}) to produce interpolated samples centered around
$\rho=0$. In view of the periodicity and smoothness of the function $v$,
this procedure is highly accurate, and it is, in fact significantly faster
and less memory intensive than the full two-dimensional spline-table
construction presented in~\cite{YingBirosZorin}---since it only requires
storage of one-dimensional tables.
\end{remark} 
\begin{remark}\label{shift}
Given point values $v(x_i)$ of a smooth and periodic function $v$ on an
equispaced grid $x_i=x_0 + i h$, samples of $v$ on a new shifted grid
$x_i^*=x_i + \delta$ can be obtained efficiently and with spectral accuracy
through use of FFTs. The algorithm proceeds as follows: 1)~Evaluation of the
FFT of the data set $v(x_i)$ to produce Fourier coefficients of $v(x)$,
2)~Multiplication of each Fourier coefficient by an appropriate exponential,
to produce the Fourier coefficients of the shifted function $v(x+\delta)$,
and 3)~Evaluation of the inverse FFT of the coefficients produced per point
2).
\end{remark} 

\begin{figure}[h]
\center
\includegraphics[scale=0.25]{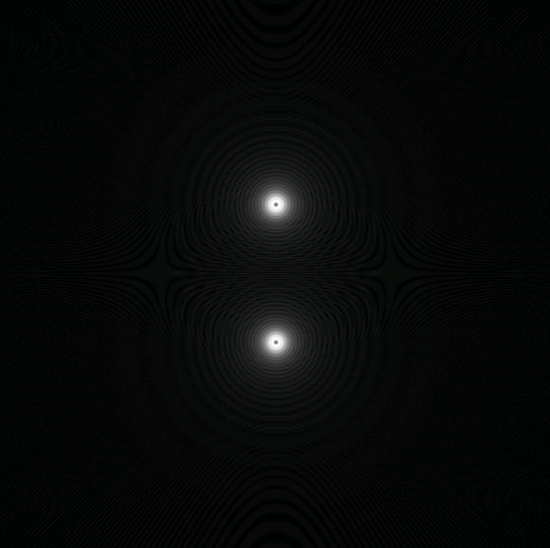}
\includegraphics[scale=0.25]{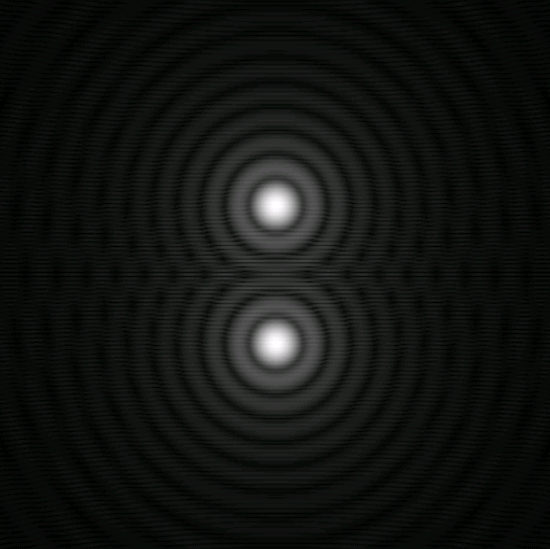}
\includegraphics[scale=0.25]{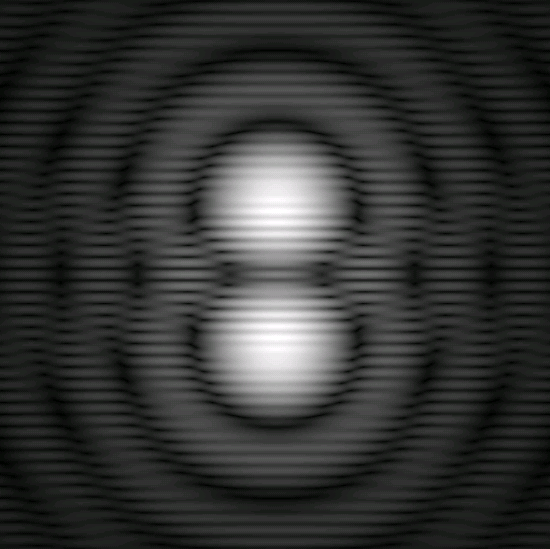}
\includegraphics[scale=0.25]{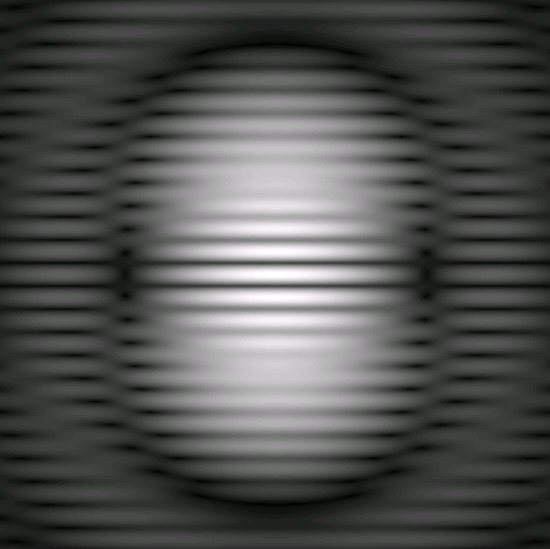}

\caption{Field diffracted by the two-hole configuration depicted in
  Figure~\ref{Young_fig}. From left to right, depiction of the diffracted
  field at observation screens located at distances of $72\lambda$,
  $576\lambda$, $1728\lambda$ and $3456\lambda$ behind the punctured
  plane. A dark-spot can be
  viewed again at the center of the illuminated circles in the left-most image by
  adequately enlarging the image.\label{Young_fringes}}.
\end{figure}

\begin{figure}
\center
\includegraphics[scale=0.4]{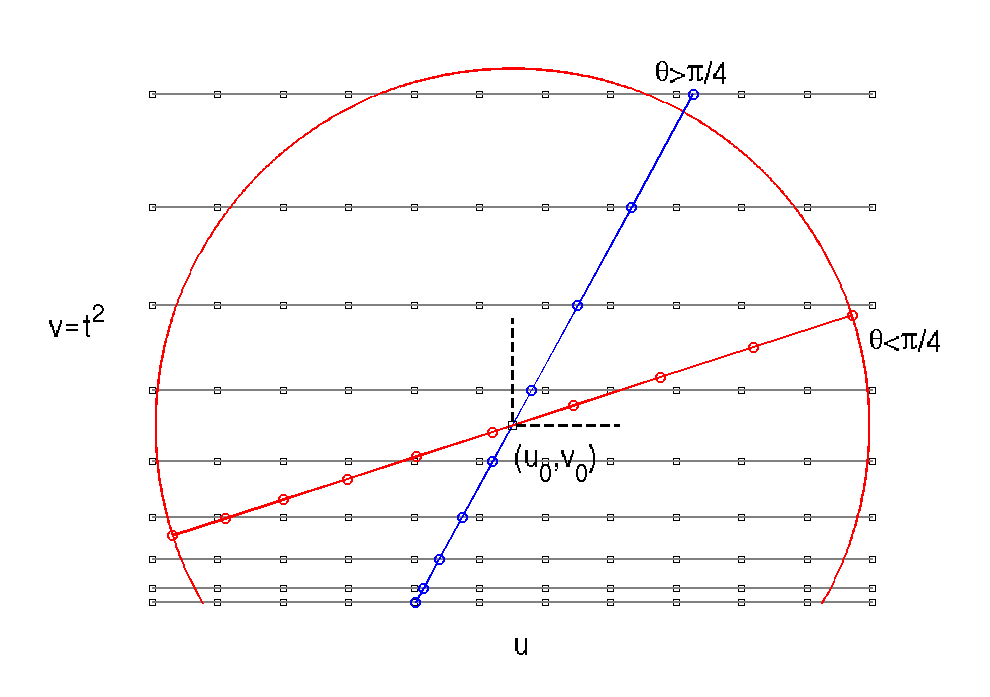}
\caption{Polar changes of variables around a point close to
the edge: quadratic sampling in the $v$ variable, requiring off-grid
interpolations for grazing angles. 
 }
  \label{polar_edge_fig}
\end{figure}

\section{High-order evaluation of edge-patch operators}\label{sec_exterior}
In this section we describe our algorithms for evaluation of the edge-patch
operators, namely the integral operators of type III, IV and~VI and the
differentiation operator~\eqref{partial2}. To do this, and in accordance
with Section~\ref{sec_POU}, we select a tensor product grid
$(u^{q,2}_{\ell},v^{q,2}_{m})$ quadratically refined in $v$, which, using
spatial mesh-sizes $h^{q,2}_u$ and $h^{q,2}_t$ in the $u$ and $t$ variables,
is given by

\begin{equation}
\left\{
\begin{array}{cc}
u^{q,2}_\ell =u^{q,2}_0+ \ell h^{q,2}_u,& \ell=1,\dots L^2_{q} \\ 
v^{q,2}_m= \left( (\frac{1}{2} + m )h^{q,2}_t\right)^2, &m=1\dots M^2_{q}.
\end{array} \right.
\end{equation}

This grid is assumed to discretize a rectangle that contains
$\mathcal{H}^q_2$; in view of the assumptions made on this set
(Section~\ref{sec_POU}) and the form of the discretization
$(u^{q,2}_{\ell},v^{q,2}_{m})$ we see that, while the edge $v=0$ is not
itself sampled by this discretization, a parallel line to it, at a distance
of $(h^{q,2}_t/2)^2$ in $(u,v)$ space, is.  

\subsection{Type III Integral (Regular)}
For any smooth function $g$ defined over the interval $[0,1]$ which vanishes
identically with all its derivatives at $x=1$, the function $g(t^2)$ can
clearly be extended as a smooth and periodic function of period 2. It
follows immediately from the identity
\begin{equation}
\int_0^1 \frac{g(x)}{\sqrt{x}}dx=\int_{-1}^1 g(t^2) dt =2\int_0^1 g(t^2)dt
\end{equation}
that the trapezoidal rule approximation
\begin{equation}\label{trap_edge}
\int_0^1\frac{g(x)}{\sqrt{x}}dx\sim \frac{2}{M}\sum\limits_{m=1}^M g\left( t_m^2\right),\quad t_m=\frac{2m+1}{2M},i=0,\dots,M-1
\end{equation}
gives rise to super-algebraic convergence. Since the patch discretization
$(u^{q,2}_\ell,v_m^{q,2})$ can be expressed in the form
$v_m^{q,2}=(t_m^{q,2})^2$, where $t_m^{q,2}=(\frac{1}{2}+m)h_t^{q,2}$, a
two-dimensional trapezoidal rule using this mesh in the set $\mathcal{H}_2^q
$ is super-algebraically convergent.

\subsection{Partial Derivatives}\label{sec_partial2}

In view of the smoothness and periodicity of the function $\phi_2(u,t^2)$, a
two-dimensional interpolation scheme based on use of FFTs along the $u$
variable and FCTs (Fast Cosine Transform) along the $t$ variable yields
spectrally convergent approximations of the function $\phi_2(u,t^2)$ and its
derivatives. Our algorithm thus evaluates the derivatives required in
equation~\eqref{partial2} by performing a direct term by term
differentiation of the resulting Fourier representations together with the
expression
\begin{equation}\label{denominator}
\frac{\partial \phi_2(u,t^2)}{\partial v}=\frac{1}{2t}\frac{\partial
}{\partial t}\left[\phi_2(u,t^2)\right].
\end{equation}

\begin{remark}\label{Remark_sqrt}
In view of the presence of the $t=\sqrt{v}$ denominator on the right-hand
side of equation~\eqref{denominator}, evaluation of partial derivatives of
the function $\phi_2(u,v)$ with respect to $v$ on the basis of a term by
term differentiation of cosine expansion of the function $\phi_2(u,t^2)$,
while yielding spectrally accurate results, is less accurate near the edge
than away from the edge---in close analogy with the well-known relative loss
of accuracy around end points in Chebyshev-based numerical
differentiation. This is why our algorithm was designed to evaluate the
composite operator $\mathbf{N}_\omega\mathbf{S}_\omega$ by first producing
the combination $\mathcal{T}_\omega\mathbf{S}_\omega$ via the rules derived
for $D_\omega^{\mathbf{T}}$ as explained in Section~\ref{sec_combined}, thus
avoiding numerical differentiation. Unfortunately, in the evaluation of the
operator $\mathbf{N}_\omega$ (which is necessary e.g. for solution of
equation equation~\eqref{NGood}), direct computation of derivatives and
associated accuracy loss does not seem to be avoidable.
\end{remark}

\subsection{Type IV Integral (Singular)}\label{type_4_sec}
As in Section~\ref{type_2_sec}, we utilize a polar change of variables to
resolve the Green's function singularity in the canonical type defined in
equation~\eqref{Iqsing_w}, thus obtaining the expression
\begin{equation}\label{edge_polar}
\mathcal{I}_2^{q,sing}[\phi_2](u_0,v_0)=\int_0^\pi 
\left(\int_{-\frac{v_0}{\sin\theta}}^\infty H(u_0,v_0,\rho,\theta)
\frac{d\rho}{\sqrt{v_0+\rho\sin\theta}}\right)d\theta,
\end{equation}
where the integrand
$H(u_0,v_0,\rho,\theta)=\phi_2(u_0+\rho\cos\theta,v_0+\rho\sin\theta)\frac{|\rho|}{R}$
is a smooth function of $\rho$ and $\theta$, which vanishes for $\rho$
larger than a certain constant $\rho_0$. While the square-root singularity
in the inner-integral can clearly be resolved to high-order by applying an
appropriate quadratic change of variable, the outer integrand in $\theta$ is
not a uniformly smooth function: as detailed in Appendix~\ref{appendix_BL},
it develops a boundary-layer as $v_0$ approaches $0$. The analysis presented
in Appendix~\ref{appendix_BL} suggests a simple and efficient method for
high-order resolution of this boundary layer---thus leading to accurate
evaluation of the integral~\eqref{edge_polar}. This methodology, which is an
integral part of our solver, is described in what follows.

The aforementioned boundary-layer integration method is based on use of the
change of variables $t=\sqrt{v_0+\rho\sin\theta}$. With this change of
variables equation~\eqref{edge_polar} becomes
\begin{equation}\label{edge_polar_nosqrt}
\mathcal{I}_2^{q,sing}[\phi_2](u_0,v_0)=\int_0^\pi I_{\rho,2}^q[\phi_2](u_0,v_0,\theta)d\theta,
\end{equation}
\begin{equation}\label{Iqrho2}
I_{\rho,2}^q[\phi_2](u_0,v_0,\theta)=\int_{0}^\infty H(u_0,v_0,\frac{t^2-v_0}{\sin\theta},\theta)dt.
\end{equation}
The integral~\eqref{Iqrho2} is evaluated with high-order accuracy by means
of a trapezoidal rule in the $t$ variable, for any $0<\theta<\pi$. In order
to capture the boundary-layer in the outer-integral
in~\eqref{edge_polar_nosqrt}, our algorithm relies on an additional changes
of variables $\theta = \alpha^2$ and  $\theta = \pi-\alpha^2$, which lead to
the expression 
\begin{equation}\label{Irho_edge_alpha} \int_0^\pi
I^q_{\rho,2}[\phi_2](u_0,v_0,\theta)d\theta =
\int_0^{\sqrt{\frac{\pi}{2}}}\left(
I^q_{\rho,2}[\phi_2](u_0,v_0,\alpha^2)-I^q_{\rho,2}[\phi_2](u_0,v_0,\pi-\alpha^2)
\right)\alpha d\alpha.
\end{equation}
In view of the analysis presented in Appendix~\ref{appendix_BL}, the
boundary layer is confined to the interval $[0,\alpha^*(v_0)]$, where
$\alpha^*(v_0)=(\frac{v_0}{d})^\frac{1}{3}$, and we therefore
decompose the $\alpha$-integral in the form
\begin{equation}\label{Irho_edge_alpha_split}
\int_0^{\sqrt{\frac{\pi}{2}}} \dots d\alpha = \int_0^{\alpha^*(v_0)} \dots d\alpha +
\int_{\alpha^*(v_0) }^{\sqrt{\frac{\pi}{2}}} \dots d\alpha.
\end{equation}
For a given error tolerance, our algorithm proceeds by applying Chebyshev
integration rules to both integrals in~\eqref{Irho_edge_alpha_split}, using
for the second integral a number of integration points that does not depend
on $v_0$, and using for the first integral a number of integration points
that grows slowly as $v_0$ tends to zero. In practice, we have found that a
mild logarithmic growth in the number of integration points suffices to give
consistently accurate results. In view of such slow required growth, and for
the sake of simplicity, the number of integration points used for evaluation
of the first integral in~\eqref{Irho_edge_alpha_split} was taken to be
independent of $v_0$ and sufficiently large to meet prescribed error
tolerances; we estimate that a minimal additional computing time results
from this practice in all of the examples considered in this paper.

\begin{remark}\label{polar_edge}
In order to apply the trapezoidal rule~\eqref{trap_edge} for evaluation the
integral of~\eqref{Iqrho2} we distinguish two cases, as illustrated in
Figure~\ref{polar_edge_fig}. For $\frac{\pi}{4}\leq \theta \leq
\frac{3\pi}{4}$ we use the sampling in $t$ provided by intersections with
the original grid underlying $\mathcal{P}^q_2$: the 1-dimensional
cubic-spline interpolation method introduced in
section~\ref{type_2_sec} can be used to efficiently interpolate
the function $H(\frac{t^2-v_0}{\sin\theta},\theta)$ at the needed
integration points. For $0\leq\theta\leq\frac{\pi}{4}$ and
$\frac{3\pi}{4}\leq \theta \leq\pi$, on the other hand, the $t$-sampling provided by the
intersections with the original grid is too coarse. In this case, we resort
to a full two-dimensional interpolation of the density $\phi_2$ (see
Remark~\ref{two_d_interp}) to interpolate to a mesh in the $t$ variable
which, away from $t=0$ has roughly the same sampling density as that in the
overall patch discretization. In practice a fixed number of discretization
points is used to discretize all of the $t$ integrals considered in the
present remark.
\end{remark}
\begin{remark}\label{two_d_interp}
The two-dimensional interpolation method for smooth functions $\phi_2(u,v)$,
which is mentioned in Remark~\ref{polar_edge}, proceeds by first performing
a two-dimensional Fourier expansion of the function $\phi_2(u,t^2)$, by
means of FFTs along the $u$ variable and FCTs along the $t=\sqrt{v}$
variable, followed zero-padding by a factor $P$ (in practice we use
$P=6$). This procedure results in a spectral approximation of $\phi_2$ (and,
by term-by-term differentiation, of its derivatives as well) on a highly
resolved two-dimensional grid. The final interpolation scheme is obtained by
building bi-cubic spline interpolations based on function values and
derivatives on each square of the refined
grid.~\cite[p. 195]{NumericalRecipes}.
\end{remark}

\begin{figure}[h]
\center
\includegraphics[scale=0.5]{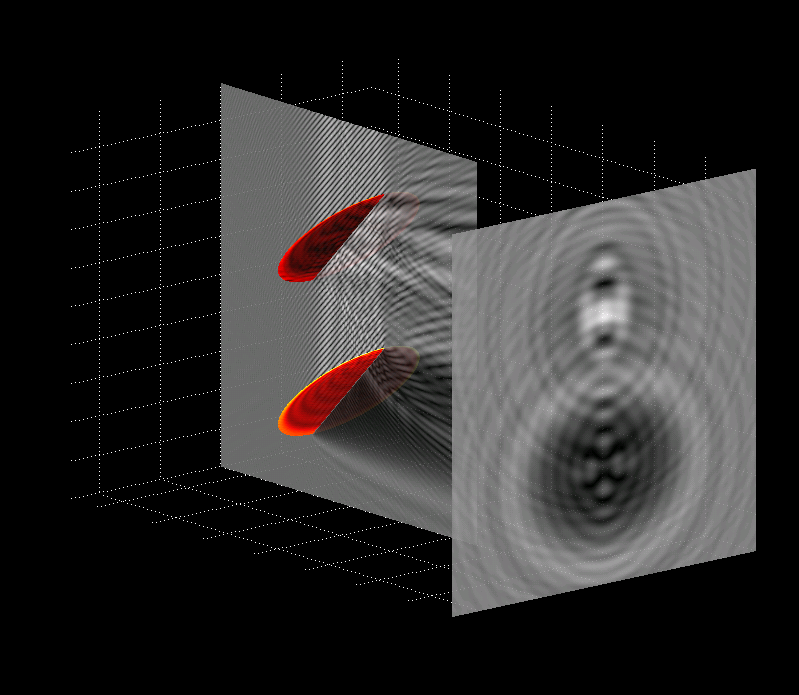}
\caption{Multiple scattering examples: Neumann problem on two parallel discs
  $24\lambda$ in diameter, illuminated at a $45$ degree angle. The coloring
  on the discs represents the values of the surface unknown $\psi$.
  \label{twoDiscs_fig}}
\end{figure}

\subsection{Type VI Integral (Principal Value)}\label{type_6_sec}
Our algorithm evaluates the principal-value edge-patch Type VI canonical
integral $\mathcal{I}^{q,pv}_2$ in a manner similar to that used for Type IV
treated in Section~\ref{type_4_sec}: introducing the local polar change of
variables around the point $\mathbf{r}$ we obtain
\begin{equation}\label{i2_new}
\mathcal{I}^{q,pv}_2[\phi_2](u_0,v_0)=\int_0^\pi
 \left(p.v. \int_\frac{-v_0}{\sin\theta}^\infty
 \frac{H^{T}(u_0,v_0,\rho,\theta)}{\rho}\frac{d\rho}{\sqrt{v_0+\rho\sin\theta}}
 \right) d\theta,
\end{equation}
where 
\begin{equation}
H^T(u_0,v_0,\rho,\theta)=\phi_2(u_0+\rho\cos\theta, v_0+\rho\sin\theta)\frac{|\rho|^3}{R^3}\frac{\mathbf{R}.\mathbf{T}}{\rho}
\end{equation}
is again, a smooth function of $\rho$ and $\theta$ which vanishes
identically for $\rho>\rho_0$.  Resorting to the quadratic change of
variables $t = \sqrt{v_0+\rho\sin\theta}$ we obtain
\begin{equation}\label{I_pv_fac}
\mathcal{I}^{q,pv}_2[\phi_2](u_0,v_0)=\int_0^\pi I^{q,pv}_{\rho,2}[\phi_2](u_0,v_0,\theta)d\theta.
\end{equation}
where the radial integral 
\begin{equation}\label{I_rho_pv_edge}
I^{q,pv}_{\rho,2}[\phi_2](u_0,v_0,\theta)=p.v. \int_0^\infty H^T(u_0,v_0,\frac{t^2-t_0^2}{\sin\theta},\theta)\frac{dt}{t^2-t_0^2}, \quad t_0=\sqrt{v_0}
\end{equation}
can be expressed in the form
\begin{equation}
I^{q,pv}_{\rho,2}=\mbox{p.v.}\int_{0}^\infty \frac{v(t^2)}{t^2-t_0^2}dt\quad , \quad
  \mbox{where $v(t)$ is smooth for $t\geq 0$ and vanishes for $t$ large enough}.
\end{equation}
Simple algebra then yields
\begin{equation}\label{pv_edge_rule}
I^{q,pv}_{\rho,2}=\mbox{p.v.}\int_{0}^\infty
v(t^2)\left\{\frac{1}{t-t_0}-\frac{1}{t+t_0}\right\}dt=\mbox{p.v.}\int_{-\infty}^\infty\frac{v(t^2)}{t-t_0}dt;
\end{equation}
clearly, the right hand side integral in equation~\eqref{pv_edge_rule} can
be evaluated with high-order accuracy by means of the trapezoidal
rule~\eqref{pvcentral}.

Using this algorithm for evaluation of the integral~\eqref{I_rho_pv_edge}
for any fixed value of $\theta$ our algorithm for evaluation of the $\theta$
integral~\eqref{I_pv_fac}, and thus~\eqref{i2_new}, is completed, as in
Section~\ref{type_4_sec}, by relying on (a)~The quadratic change of variable
$\theta=\alpha^2$ and (b)~The boundary-layer
split~\eqref{Irho_edge_alpha_split}.

\section{Parameter Selection}\label{param_sel}
A number of parameters are implicit in the algorithm laid out in
Section~\ref{numer_framework}, including parameters that relate to the
overall surface patching and discretization strategies described in
Section~\ref{sec_POU} as well as parameters that arise in the polar
integration rules introduced in
Sections~\ref{type_2_sec},~\ref{type_5_sec},~\ref{type_4_sec}
and~\ref{type_6_sec}. Clearly, the values of such parameters have an impact
on both, accuracy for a given discretization density, as well as computing
time for a given accuracy tolerance. A degree of experimentation is
necessary to produce an adequate selection of such parameters for a given
problem. Without entering a full description of the choices inherent in our
own implementations, in what follows we provide an indication of the
strategies we have used to select two types of parameters, namely, (a)~The
width of the edge patches (see Figure~\ref{discPOU}), and (b)~The number of
discretization points used in both the radial and angular directions for
each polar integration problem for which the corresponding floating
partition of unity does not vanish at the open edge, as illustrated in
Figure~\ref{polar_edge_fig}. Similar (but simpler) considerations apply to
other parameters, such as width of floating partitions of unity, extents of
overlap between patches, etc.

With respect to point (a)~above we note that, for scattering solutions to be
obtained with a fixed accuracy tolerance, the discretization densities must
be increased as frequencies are increased, and, thus, the width of the edge
patches can be decreased accordingly---in such a way that the number of
discretization points in the $v$ direction for each one of the edge patches
is kept constant. This strategy is crucial for efficiency, since the edge
patches require use of the two-dimensional interpolation method mentioned in
Remark~\ref{two_d_interp}, which is significantly more costly than the
corresponding one-dimensional interpolation method used in the interior
patches. Use of constant number of $v$-discretization points within
shrinking edge patches for increasing frequencies thus enables
fixed-accuracy evaluation of edge-patch integrals with an overall computing
cost that is not dominated by the edge-patch two-dimensional interpolation
procedure.

Concerning point (b)~above, in turn, as mentioned in
Remark~\ref{polar_edge}, we make use of a fixed number of equispaced
integration points in the scaled radial variable $t$ (see
equation~\eqref{edge_polar_nosqrt}) for all values of the angular variable
$\theta$. In practice, we select the number of $t$-integration points to
equal the maximum value $N_t$ of the numbers $N_u$ and $N_v$ of points in
the $u$-$v$ discretization mesh that are contained in the $\theta =0$ and
$\theta =\pi/2$ lines, respectively, and which lie within the support of the
corresponding floating POU.  In order to preserve the wavelength sampling in
the angular integral~\eqref{edge_polar_nosqrt}, finally, the two integrals
on the right-hand-side of equation~\eqref{Irho_edge_alpha_split} are
evaluated on the basis of the Clenshaw-Curtis quadrature
rule~\cite{NumericalRecipes} using an $\alpha$ discretization mesh
containing $\frac{\pi}{2} N_t$ points---since the length of half a
circumference equals $\frac{\pi}{2}$ times its diameter.

\section{Numerical Results}\label{sec_numerical}
In this section, we present results obtained by means of a C{}\verb!++!
implementation of the algorithm outlined in Section~\ref{nystrom},
incorporating the canonical operator decompositions introduced in
Sections~\ref{sec_S} through~\ref{sec_combined} for the operators
$\mathbf{S}_\omega$, $\mathbf{N}_\omega$ and
$\mathbf{N}_\omega\mathbf{S}_\omega$, together with the high-order
integration rules put forth in Sections~\ref{sec_interior}
and~\ref{sec_exterior} and the iterative linear algebra solver GMRES. Errors
reported were evaluated through comparisons with highly-resolved numerical
solutions. Computation times correspond to single-processor runs (on a
2.67GHz Intel core), \emph{without use of the acceleration methods or
parallelization.} As mentioned in Section~\ref{efficiency}, application of
the acceleration method~\cite{BrunoKunyansky} in the present context does
not present difficulties; such extension will be considered in forthcoming
work.

\begin{table}
\begin{center}
\begin{tabular}[c]{  c  c c  c  c }
\hline $N$ & Dir($\mathbf{S}_\omega$) & Dir($\mathbf{N}_\omega\mathbf{S}_\omega$)& Neu($\mathbf{N}_\omega$) & Neu($\mathbf{N}_\omega\mathbf{S}_\omega$) \\ 
\hline
$16\times16 + 2\times24\times16$ & $2.4\times 10^{-4}$ & $2.5\times10^{-4}$ & $5.0\times
10^{-4}$&$2.6\times 10^{-4}$ \\
$32\times32 + 2\times48\times32$ &$4.8\times 10^{-6}$ & $4.8\times 10^{-6}$ & $5.3\times 10^{-6}$  &
$5.2\times 10^{-6}$ \\
$64\times64 + 2\times98\times64$ &$4.7\times 10^{-8}$ & $9.7\times10^{-8}$ & $4.9\times 10^{-8}$ &
$5.1\times 10^{-8}$ \\
\hline
\end{tabular} 
\caption{Scattering by a disc of diameter $3\lambda$, similar to the
  corresponding $24\lambda$ simulation depicted in
  Figure~\ref{disc_fig}. Maximum errors in the acoustic field on the square
  projection plate shown in the figure. This table demonstrates spectral
  convergence for all the formulations considered: doubling the
  discretization density results in orders-of-magnitude decreases in the
  numerical error. (The notation $Q_1\times m_1\times n_1 + Q_2\times
  m_2\times n_2$ indicates that a number $Q_1$ of patches containing
  $m_1\times n_1$ discretization points together with a number $Q_2$ of
  patches containing $m_2\times n_2$ discretization points were used for the
  corresponding numerical solution.)
\label{spectral_table}}
\end{center}
\end{table}

\begin{table}[h!]
\begin{center}
\begin{tabular}[c] {  c  c | c  c  c | c  c c  }
\multicolumn{2}{c }{} & \multicolumn{3}{ c }{ Dir($\mathbf{S}_\omega$)} & \multicolumn{3}{ c }{Dir($\mathbf{N}_\omega\mathbf{S}_\omega$)}  \\
  \hline
 Disc Size & Unknowns  &  It. & Time & $\epsilon_r$&  It. & Time & $\epsilon_r$  \\
\cline{1-8}
\cline{1-8}
\hline
\hline
$3\lambda$ & $4096$ &$6$&  $58\mbox{s}$ & $1.0\times 10^{-4}$ &6 &$5\mbox{m}20\mbox{s}$ & $1.4\times 10^{-4}$ \\
\hline
$6\lambda$ & $10240$ &$9$&  $3\mbox{m}5\mbox{s}$ & $8.2\times 10^{-5}$ &6 &$11\mbox{m}14\mbox{s}$ & $5.4\times 10^{-5}$ \\
\hline
$12\lambda$ & $28672$ &$13$&  $15\mbox{m}21\mbox{s}$ & $1.2\times 10^{-4}$ &7 &$36\mbox{m}31\mbox{s}$ & $3.7\times 10^{-4}$ \\
\hline
$24\lambda$ & $90112$ &$18$&  $2\mbox{h}30\mbox{m}$ & $2.8\times 10^{-4}$ &7 &$3\mbox{h}41\mbox{m}$ & $4.1\times 10^{-4}$ \\
\multicolumn{8}{c }{}\\ 
\multicolumn{2}{c }{} & \multicolumn{3}{ c }{ Neu($\mathbf{N}_\omega$)} & \multicolumn{3}{ c }{Neu($\mathbf{N}_\omega\mathbf{S}_\omega$)}  \\
  \hline
 Disc Size & Unknowns  &  It. & Time & $\epsilon_r$&  It. & Time & $\epsilon_r$  \\
\cline{1-8}
\cline{1-8}
\hline
\hline
$3\lambda$ & $4096$ &$16$&  $9\mbox{m}21\mbox{s}$ & $1.3\times 10^{-4}$ &6 &$5\mbox{m}50\mbox{s}$ & $1.3\times 10^{-4}$ \\
\hline
$6\lambda$ & $10240$ &$28$&  $36\mbox{m}31\mbox{s}$ & $2.\times 10^{-4}$ &6 &$11\mbox{m}36\mbox{s}$ & $5.8\times 10^{-5}$ \\
\hline
$12\lambda$ & $28672$ &$49$&  $2\mbox{h}43\mbox{m}$ & $1.7\times 10^{-4}$ &7 &$37\mbox{m}04\mbox{s}$ & $5.2\times 10^{-4}$ \\
\hline
$24\lambda$ & $90112$ &$80$&  $21\mbox{h}10\mbox{m}$ & $1.9\times 10^{-4}$ &7 &$3\mbox{h}51\mbox{m}$ & $6.1\times 10^{-4}$ \\
\end{tabular}

 \caption
{Iteration numbers and computing times for the problem of scattering by a
disc at normal incidence. Top: Dirichlet problem. Bottom: Neumann
problem. In each case, use of the second-kind combined operator
$\mathbf{N}_\omega\mathbf{S}_\omega$ gives rise to significantly smaller
iteration numbers than the corresponding first kind formulation. In the case
of the Neumann problem, the reduction in iteration numbers results in
substantially improved computing times. Note: all reported computing times
correspond to \emph{non-accelerated single-processor runs}. Dramatic
reductions in computing times would result from use of the acceleration
method~\cite{BrunoKunyansky}---see e.g. the recent
contribution~\cite{BrunoEllingTurc} for the closed-surface
case.\label{Disc_Table} }
\end{center}
\end{table}

\subsection{Spectral convergence}
We demonstrate the spectral properties of our algorithm through an example
concerning a canonical geometry, namely, the unit disc
\begin{equation}
x^2+y^2\leq 1, \quad z=0.
\end{equation}
For this surface we utilize three coordinate patches (see
Figure~\ref{discPOU}), including a large central patch given by equations
$\left\{x(u,v)=u, \quad y(u,v)=v, \quad z(u,v)=0\right\} $, and two edge
patches parametrized by the equations $\left\{x(u,v)=(1-v)\cos u,\quad
y(u,v)=(1-v)\sin u,\quad z(u,v)=0\right\}$ for values of $u$ and $v$ in
adequately chosen intervals.  The two edge-patches overlap as illustrated in
Figure~\ref{discPOU}, and their width is defined by the range of the $v$
variable, which, in accordance with Section~\ref{param_sel}, is reduced as
the frequency increases.  With reference to equation~\eqref{asymp}, the
integral weight is set to $\omega=\sqrt{1-x^2-y^2}$.


\begin{figure}[h]
\center
\includegraphics[scale=0.3]{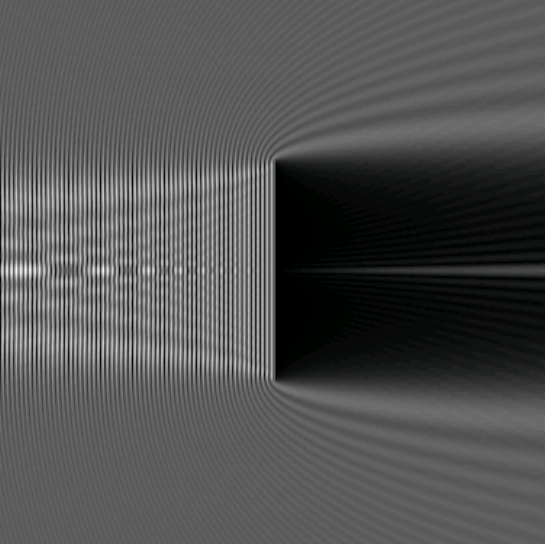}
\includegraphics[scale=0.3]{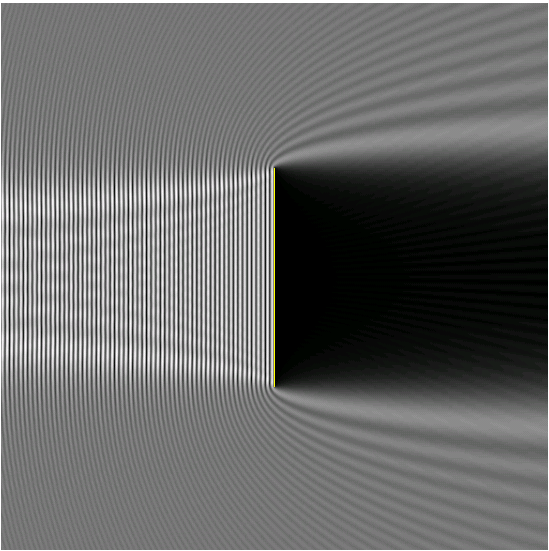}
\caption{Poisson-spot phenomenon. Left: cross-sectional view on a of the
diffraction pattern produced by a disc $24\lambda$ in diameter in three
dimensional space (Dirichlet problem, normal incidence). Right: Diffraction
by an arc of length $24\lambda$ in two dimensional space (Dirichlet problem,
normal incidence). Note that only in the three-dimensional case does a
``Poisson cone'' and corresponding ``Poisson spot'' develop in the shadow
region.\label{poisson_fig}}
\end{figure}

The sound-soft (Dirichlet) problem can be solved by means of either the
first-kind equation~\eqref{SGood}, or the second-kind
equation~\eqref{NSGoodD} which, in what follows, are called
Dir($\mathbf{S}_\omega$) and Dir($\mathbf{N}_\omega\mathbf{S}_\omega$),
respectively. The sound-hard (Neumann) problem, similarly, can be tackled by
means of either the first-kind equation~\eqref{NGood} or the second-kind
equation~\eqref{NSGoodN}; we call these equations Neu($\mathbf{N}_\omega$)
and Neu($\mathbf{N}_\omega \mathbf{N}_\omega$),
respectively. Table~\ref{spectral_table} demonstrates the high-order
convergence of the solutions produced by our implementations for each one of
these equations on a disc of diameter $3\lambda$; clearly errors decrease by
orders of magnitude as a result of a mere doubling of the discretization
density.

 \begin{table}[h!]
\begin{center}
\begin{tabular}[c] {  c  c | c  c  c | c  c c  }
\multicolumn{2}{c }{} & \multicolumn{3}{ c }{ Dir$(\mathbf{S}_\omega)$} & \multicolumn{3}{ c }{Dir$(\mathbf{N}_\omega\mathbf{S}_\omega)$}  \\
  \hline
 Spherical Cavity Size & Unknowns  &  It. & Time & $\epsilon_r$&  It. & Time & $\epsilon_r$  \\
\cline{1-8}
\cline{1-8}
\hline
\hline
$3\lambda$ & $9344$ &$17$&  $7\mbox{m}16\mbox{s}$ & $1.8\times 10^{-4}$ &13 &$1\mbox{h}18\mbox{m}$ & $4.5\times 10^{-4}$ \\
\hline
$9\lambda$ & $84096$ &$39$&  $4\mbox{h}06\mbox{m}$ & $2.1\times 10^{-4}$ &24 &$13\mbox{h}20\mbox{m}$ & $2.9\times 10^{-4}$ \\
\hline
$18\lambda$ & $336384$ &$65$&  $57\mbox{h}48\mbox{m}$ & $4.0\times 10^{-4}$ &43 &$124\mbox{h}$ & $1.4\times 10^{-4}$ \\
\hline
\multicolumn{8}{c }{}\\ 
\multicolumn{2}{c }{} & \multicolumn{3}{ c }{ Neu$(\mathbf{N}_\omega)$} & \multicolumn{3}{ c }{Neu$(\mathbf{N}_\omega\mathbf{S}_\omega)$}  \\
  \hline
 Spherical Cavity Size & Unknowns  &  It. & Time & $\epsilon_r$&  It. & Time & $\epsilon_r$  \\
\cline{1-8}
\cline{1-8}
\hline
\hline
$3\lambda$ & $9344$ &$57$& 1\mbox{h}24\mbox{m} & $1.2\times 10^{-4}$&13 &$1\mbox{h}20\mbox{m}$ & $8.8\times 10^{-4}$ \\
\hline
$9\lambda$ & $84096$ &$243$& 52\mbox{h}14\mbox{m} & $2.\times 10^{-4}$ &24 &$13\mbox{h}21\mbox{m}$ & $5.6\times 10^{-4}$ \\
\hline
$18\lambda$ & $336384$ &$>600$& - &- &43 &$124\mbox{h}$ & $3.1\times 10^{-4}$ \\
\hline
\end{tabular}

 \caption{ Iteration numbers and computing times for the problem of
scattering by the spherical cavity defined by equation~\eqref{cavity_param}
and depicted in Figures~\ref{cavity_fig} and~\ref{cavity_fig2}. Top:
Dirichlet problem. Bottom: Neumann problem. Reductions in numbers of
iterations and computing times occur as detailed in the caption of
Table~\ref{Disc_Table}, but, owing to the rich multiple scattering phenomena
that arise within the cavity, the iteration numbers are significantly
higher, in all cavity cases, than those required for the corresponding disc
problems. \label{Cavity_Table}}
\end{center}
\end{table}

\subsection{Solver performance under various integral formulations} 
In this section we demonstrate the performance of the open-surface solvers
based on use of the operators Dir($\mathbf{S}_\omega$) and
Dir($\mathbf{N}_\omega\mathbf{S}_\omega$) for the Dirichlet problem, as well
as the operators Neu($\mathbf{N}_\omega$) and Neu($\mathbf{N}_\omega
\mathbf{S}_\omega$) for the Neumann problem.  We base our demonstrations on
two open surfaces: a disc and a spherical cavity defined by
\begin{equation}\label{cavity_param}
x^2 + y^2 +z^2 =1, \quad z>\cos( \theta_0),
\end{equation}
where $\theta_0$ denotes the cavity aperture. For the examples discussed
here we set $\theta_0=\frac{3\pi}{4}$, and we made use of the weight
function $\omega=\sqrt{z-z_0}$ where $z_0 =\cos( \theta_0)$.

For both geometries, computational times and accuracies at increasingly
large frequencies are reported in Tables~\ref{Disc_Table}
and~\ref{Cavity_Table}. In all tables the acronym It. denotes the number of
iterations required to achieve a relative error (in a screen placed at some
distance from the diffracting surface) equal to ``$\epsilon_r$'' (relative
the the maximum field value on the screen), and ``Time'' denotes the total
time required by the solver to evaluate the solution.  As can be seen from
these tables, the equation Neu($\mathbf{N}_\omega$) requires very large
number of iterations for the higher frequencies. The computing times
required by the low-iteration equation
Neu($\mathbf{N}_\omega\mathbf{S}_\omega$) are thus significantly lower than
those required by Neu($\mathbf{N}_\omega$). The situation is reversed for
the Dirichlet problem: although the equation
Dir($\mathbf{N}_\omega\mathbf{S}_\omega$) requires fewer iterations than
Dir($\mathbf{S}_\omega$), the total computational cost of the low-iteration
equation is significantly higher in this case---since the application of the
operator in Dir($\mathbf{S}_\omega$), which, fortunately, suffices for the
solution of the Dirichlet problem, is substantially less expensive than the
application of the operator in Dir($\mathbf{N}_\omega\mathbf{S}_\omega$). As
it happens, at high-frequency, the bulk of the computational time used by
our solver is spent on interior-patch work: application of the acceleration
method of~\cite{BrunoKunyansky} (see also~\cite{BrunoEllingTurc}) would
therefore reduce dramatically overall computing times for high-frequency
problems.

Figures~\ref{disc_fig} and~\ref{disc_fig2} display three dimensional
renderings of patterns of diffraction by the disc under normal incidence
with Neumann boundary condition, and under horizontal incidence with
Dirichlet boundary condition. Corresponding images for the spherical-cavity
problem are presented in Figures~\ref{cavity_fig} and~\ref{cavity_fig2};
note the interesting patterns of multiple-scattering and caustics that arise
in the cavity interior.

\begin{figure}[h]
\center
\includegraphics[scale=0.5]{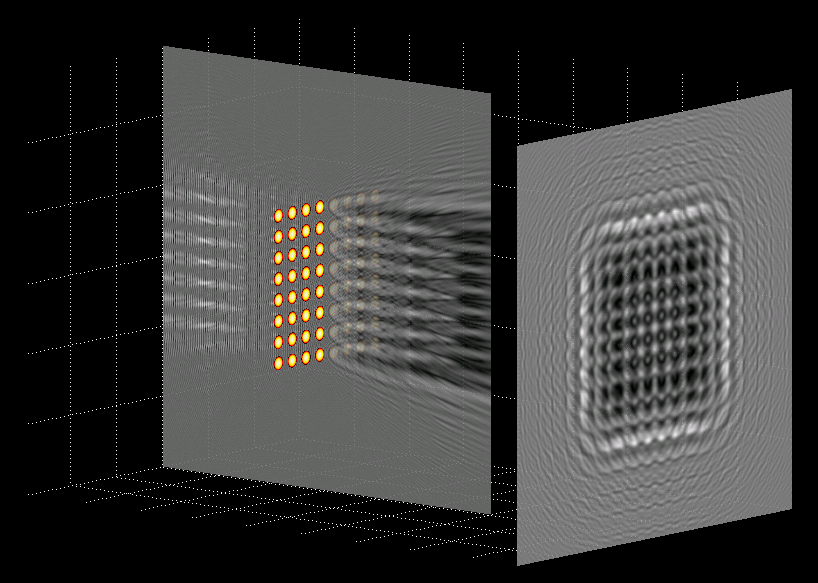}
\caption{Dirichlet problem on an array of $8\times8$ discs of diameter
  $6\lambda$. (Overall diameter: 96.6$\lambda$; 192 patches used.) The coloring on the discs represent the values of the surface unknown $\varphi$.
  \label{array_fig}}
\end{figure}
\subsection{Miscellaneous examples}
This section presents a variety of results produced by the open-surface
solver introduced in this paper, including demonstration of well known
effects such as the Poisson spot, and applications in classical contexts
such as that provided by the Young experiment.
\subsubsection{Poisson spot}
As mentioned in the Introduction, the experimental observation of a bright
area in the shadow of the disc, the famous Poisson spot, provided one of the
earliest confirmations of the wave-theory models of light.  The Poisson spot
is clearly visible in the diffraction patterns presented in
Figure~\ref{disc_fig} and the left portion of Figure~\ref{poisson_fig}. The
left portion of Figure~\ref{poisson_fig} displays a slice of the total field
around the disc along the $x-z$ plane, which gives a better view of the
Poisson-spot phenomenon: the ``Poisson cone'' is clearly visible in this
figure. Interestingly, this phenomenon does not occur in the two-dimensional
case. This is demonstrated in the image presented on the right portion of
Figure~\ref{poisson_fig}: the two-dimensional diffraction pattern arising
from the flat unit strip (which was obtained by the solver presented
in~\cite{BrunoLintner2}) gives rise to a dark dark shadow area which does
not contain a diffraction spot.
\begin{figure}[h]
\center
\includegraphics[scale=0.5]{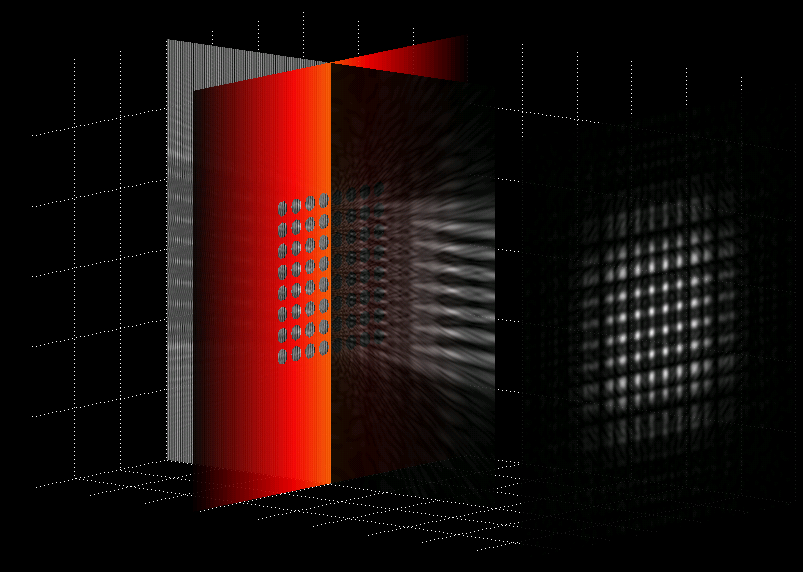}
\caption{Neumann problem for an array of $8\times8$ circular apertures of
diameter $6\lambda$. The diffracted field depicted in this figure was
produced by means of Babinet's principle from the diffraction pattern
displayed in Figure~\ref{array_fig}. As in Figure~\ref{aperture_fig}, the coloring on the plane is introduced for  visual quality, and does not represent any physical quantity.
 \label{array_Babinet}}
\end{figure}

\subsection{Babinet's principle, apertures and Young's experiment\label{apertures}}
For a flat open surface $\Gamma$ contained in a plane $\Pi$ one may consider
the corresponding problem of diffraction by the complement $\Gamma^c = \Pi
\setminus \Gamma$ of $\Gamma$ within $\Pi$. As is well known, the
diffraction pattern resulting from $\Gamma^c$ can be computed easily, by
means of the Babinet principle (see Appendix~\ref{appendix_babinet} and, in
particular, equation~\eqref{babinet_form}), from a corresponding diffraction
pattern associated with the surface $\Gamma$. (For ease of reference, a
derivation of the Babinet principle for scalar waves is presented in
Appendix~\ref{appendix_babinet}.) In what follows we present three
applications of the Babinet principle, namely, the diffraction by a circular aperture, the Young phenomenon, and diffraction across an array of apertures (in
Section~\ref{arrays}).

As our first application of Babinet's principle, in
Figure~\ref{aperture_fig} we present the field diffracted by a circular
aperture which is $24\lambda$ in diameter. The incident field for this image
was taken to be a point source located at the point $(0,0,-10)$, outside the
region displayed on the figure. In Figure~\ref{aperture_fringes}, we display
the total field on screens located behind the aperture at varying distance
from the punctured plane. Interestingly, under some configurations a dark
spot appears in the center of the bright area, in full accordance with
Arago's prediction that a 'Poisson shadow' must exist.  As our
second application of Babinet's principle, in Figure~\ref{Young_fig} we
present the field diffracted by a pair of nearby circular holes in an
otherwise perfectly sound-hard plane, under normal plane-wave incidence:
this is a setup of the classical Young experiment. This diffraction pattern
was produced, by means of Babinet's principle, from a corresponding solution
of a Dirichlet problem for two coplanar discs in space. As in Young's
experiment, interference fringes arise: these can be seen clearly on the
right-most image in Figure~\ref{Young_fringes}. Again, sharp dark spots at
the center of the circular illuminated areas can be seen in the leftmost
image in Figure~\ref{Young_fringes}.


\subsubsection{Arrays of scatterers/apertures\label{arrays}}
Geometries consisting of a number of disjoint open-surface scattering bodies
can be treated easily by the solver introduced in this paper---since the
decomposition in patches inherent in equation~\eqref{patch_union} is not
restricted to sets of patches representing a connected surface. The solution
for the two-disc diffraction problem presented in the previous section, for
example, was obtained in this manner. In what follows we provide a few
additional test cases involving composites of open surfaces.

Figure~\ref{twoDiscs_fig} presents the solution of a problem of scattering
by two parallel discs illuminated at an angle of $\frac{\pi}{4}$, with
Neumann boundary conditions. The beam reflected by the bottom disc, which
can clearly be traced onto the upper disc, gives rise to a bright area in
the projection screen behind that disc. Two final examples concern an array
of 64 discs and the corresponding array of 64 circular apertures on a
plane---where each disc is $6\lambda$ in diameter and the discs are
separated by $3\lambda$ spacings, for a total array diameter of
$96.6\lambda$. The corresponding diffracted fields are presented in
Figures~\ref{array_fig} and~\ref{array_Babinet}. The solution of the
Dirichlet problem for the 64 disc array was obtained in $23$ GMRES
iterations on a $192$ patch geometry representation.

\section{Conclusions}
We have introduced a new set of integral equations and associated high-order
numerical algorithms for the solution of scalar problems of diffraction by
open surfaces. The new open-surface solvers are the first ones in the
literature that produce high-order solutions in reduced number of GMRES
iterations for general (smooth) open surfaces and arbitrary frequencies.
The second-kind formulation Neu($\mathbf{N}_\omega \mathbf{S}_\omega$) is
highly beneficial in the context of the Neumann problem, as it requires
computing times that are orders-of-magnitude shorter than those required by
the alternative hypersingular formulation N($\mathbf{N}_\omega$). Such gains
do not occur for the Dirichlet problem: the proposed solver produces
high-order solutions to Dir($\mathbf{S}_\omega$) in very short computational
times, and the gains in iteration numbers that result from use of the
formulation Dir($\mathbf{N}_\omega \mathbf{S}_\omega$) do not suffice to
compensate for the significantly higher cost required for evaluation of the
operator $\mathbf{N}_\omega$.  With appropriate parallelization and
acceleration, fast and accurate solutions should be achievable by these
algorithms for very large structures, thus providing a robust numerical
workbench for the solution of classical diffraction problems by open
surfaces.

\noindent
{\bf Acknowledgments.} The authors gratefully acknowledge support from
  AFOSR, NSF, JPL and the Betty and Gordon Moore Foundation.

\appendix
\section{Expression of the operator $\mathbf{N}_\omega$ in terms of
  tangential derivatives}\label{Appendix_N} In this section, we provide a
proof of Lemma~\ref{Ndec}.
\begin{proof}
It suffices to show that the operators on the left- and right-hand sides of
equation~\eqref{Ndec_eq} coincide when applied to any smooth function $\psi$
defined on $\Gamma$. Following the derivation in~\cite{ColtonKress1} for the
closed-surface case, we define the weighted double layer operator
\begin{equation}\label{DomegaDef}
\mathbf{D}_\omega[\psi](\mathbf{r})=\int_\Gamma \frac{\partial G_k(\mathbf{r},\mathbf{r}')}{\partial \mathbf{n}_\mathbf{r}'}\psi(\mathbf{r}') \omega(\mathbf{r}') dS',
\end{equation}
and we evaluate the limit of its gradient as $\mathbf{r}$ tends to
$\Gamma$. For $\mathbf{r}$ outside $\Gamma$,~\eqref{DomegaDef} can be
expressed in the form
\begin{equation}
\mathbf{D}_\omega[\psi](\mathbf{r})=-\mbox{div} \int_\Gamma G_k(\mathbf{r},\mathbf{r}') \psi(\mathbf{r}') \omega(\mathbf{r}') \textbf{n}_{\mathbf{r}'} dS'.
\end{equation}
Thus, using the identity $$\mbox{curl curl} A=-\Delta A +\nabla\mbox{div
}A,$$ we obtain
\begin{equation}\label{gradD}
\begin{split}
\nabla \mathbf{D}_\omega[\psi](\mathbf{r})=k^2\int_\Gamma
G_k(\mathbf{r},\mathbf{r}')\psi(\mathbf{r}')\omega(\mathbf{r}')\textbf{n}_{\mathbf{r}'}dS'\\
-\mbox{curl curl}\int_\Gamma G_k(\mathbf{r},\mathbf{r}')\psi(\mathbf{r}')
\mathbf{n}_{\mathbf{r}'}dS'.
\end{split}
\end{equation}
But, for $\mathbf{r}$ outside $\Gamma$ we have
\begin{equation}
\begin{split}
\mbox{curl}\int_\Gamma G_k(\mathbf{r},\mathbf{r}')\psi(\mathbf{r}')\textbf{n}_{\mathbf{r}'}dS'=\\ \int_\Gamma\left[\textbf{n}_{\mathbf{r}'},\psi(\mathbf{r}')\omega(\mathbf{r}')\nabla^s_{\mathbf{r}'}G_k(\mathbf{r},\mathbf{r}')\right]dS',
\end{split}
\end{equation}
where $[\,\cdot\, ,\cdot\,]$ and $\nabla^s$ denote the vector product and
surface gradient operator, respectively. Integrating by parts the surface
gradient (see e.g.~\cite[ eq. (2.2)]{ColtonKress1}) and noting that the
boundary terms vanish in view of the presence of the weight $\omega$, we
obtain
\begin{equation}
\begin{split}\label{curlInt}
\mbox{curl}\int_\Gamma
G_k(\mathbf{r},\mathbf{r}')\psi(\mathbf{r}')\textbf{n}_{\mathbf{r}'}dS'=\\-\int_\Gamma
G_k(\mathbf{r},\mathbf{r}')\left[\textbf{n}_{\mathbf{r}'},\nabla^s\left(\psi(\mathbf{r}')\omega(\mathbf{r}')\right)\right]dS'.
\end{split}
\end{equation}
In the limit as $\mathbf{r}$ tends to an interior point in $\Gamma$ we
therefore obtain the expression
\begin{equation}\label{grad}
\begin{split}
\nabla \mathbf{D}_\omega[\psi](\mathbf{r})=k^2\int_\Gamma
G_k(\mathbf{r},\mathbf{r}')\psi(\mathbf{r}')\omega(\mathbf{r}')\textbf{n}_{\mathbf{r}'}dS'\\
+\mbox{p.v.}
\int_\Gamma\left[\nabla_{\mathbf{r}}G_k(\mathbf{r},\mathbf{r}'),\left[\textbf{n}_{\mathbf{r}'},\nabla^s\left(
\psi(\mathbf{r}')\omega(\mathbf{r}')\right)\right]\right]dS'\quad \quad
(\mathbf{r}\in\Gamma),
\end{split}
\end{equation}
in terms of a principal value integral, for the surface values of the
gradient of the double layer operator $\mathbf{D}_\omega$. Taking the scalar
product with $\textbf{n}_{\mathbf{r}}$ on both sides of~\eqref{grad} now
yields the desired result: equation~\eqref{Ndec_eq}.

\end{proof}

\section{Boundary-layer character of the inner integral in equation~\eqref{edge_polar}}\label{appendix_BL}
In order to demonstrate the difficulties inherent in the numerical
evaluation of the outer integral in equation~\eqref{edge_polar} we consider
the integration problem
\begin{equation}\label{demo_prob}
\int_0^\pi \tilde I_\rho(v_0,\rho_0,\theta)d\theta,
\end{equation}
in which the $(u_0,v_0)$-dependent inner integral in~\eqref{edge_polar} is
substituted by the $v_0$-dependent integral
\begin{equation}\label{BL_exp}
\tilde I_\rho(v_0,\rho_0,\theta)=\int_{-\frac{v_0}{\sin\theta}}^\infty \tilde
H_{\rho_0}(\rho,\theta)\frac{d\rho}{\sqrt{v_0 +\rho\sin\theta}}\quad,\quad 0\leq
\theta \leq\pi.
\end{equation}
Here
\begin{equation}
\tilde{H}_{\rho_0}(\rho,\theta)=\left\{ \begin{array}{ll} 1, & |\rho| < \rho_0\\ 0, &
  |\rho| \geq  \rho_0, \end{array}\right.
\end{equation}
so that, in the present example, $\rho_0$ is the polar-integration
radius. (The inner integral in~\eqref{edge_polar} varies smoothly with
$u_0$, and, thus, the $u_0$ dependence does not need to be built into the
present analogy.)  The integral $\tilde{I}_\rho(v_0, \rho_0,\theta)$ is
given by
\begin{equation}\label{BL_exp}
\tilde I_\rho(v_0,\rho_0,\theta) = \left\{\begin{array}{ll} \frac{2\rho_0}{\sqrt{v_0 +
\rho_0\sin\theta} + \sqrt{v_0-\rho_0\sin\theta} } & \quad\mbox{if}\quad \sin\theta
\leq \frac{v_0}{\rho_0} \\ &\\ 2\frac{\sqrt{\frac{v_0}{\sin\theta} +
\rho_0}}{\sqrt{\sin\theta}} &\quad\mbox{if}\quad\sin\theta > \frac{v_0}{\rho_0}.
\end{array} \right. 
\end{equation}
Clearly, as $v_0$ tends to zero, $\tilde{I}(v_0,\rho_0,\theta)$ becomes
increasingly singular (as demonstrated in the left portion of
Figure~\ref{BLRaw_Fig}): in view of the last equation on the right hand side
of~\eqref{BL_exp}, we have
\begin{equation}\label{I_rho_asym}
\lim_{v_0\rightarrow 0}
\tilde{I}_\rho(v_0,\rho_0,\theta)=\left(\frac{1}{\sqrt{\sin\theta}}\right).
\end{equation} 
To treat the singularity in~\eqref{I_rho_asym} we introduce quadratic
changes of variables in the $\theta$ integration in
equation~\eqref{demo_prob}---of the form $\theta=\alpha^2$ in the interval
$[0,\frac{\Pi}{2}]$ and $\theta=\Pi-\alpha^2$ in the interval
$[\frac{\pi}{2},\pi]$). As a result of these operations we obtain bounded
integrands: for example, the integrand resulting from the first of these
changes of variables is $\tilde{J}(v_0,\rho_0,\alpha)= \alpha
\tilde{I}(v_0,\rho_0,\alpha^2)$, which is a bounded function of $\alpha$. This
integrand is depicted on the right portion of Figure~\ref{BLRaw_Fig};
clearly $\tilde{J}(v_0,\rho_0,\alpha)$ develops a boundary layer as $v_0$ tends
to zero.

The two changes of variables mentioned above result in integrals over the
domain $[0,\sqrt{\frac{\pi}{2}}]$), and in both cases boundary layers result
at and around $\alpha =0$. To resolve these boundary layers we decompose the
integration interval into two sub-intervals, namely $[0,\alpha^*(v_0,\rho_0)]$ and
$[\alpha^*(v_0,\rho_0),\sqrt{\frac{\pi}{2}}]$.  Here, for a given value of $v_0$,
the point $\alpha^*(v_0,\rho_0)$ is chosen to lie to right of the coordinate for
which the peak occurs in the right portion of Figure~\ref{BLRaw_Fig}, in
such a way that the slope of the function $\tilde{J}(v_0,\rho_0,\alpha)$ as a
function of $\alpha$ at $\alpha=\alpha^*(v_0,\rho_0)$ remains constant as $v_0$
approaches zero---with a slope that equals a certain user-prescribed
constant value. In practice we have found that the integral to the right of
the point $\alpha=\alpha^*(v_0,\rho_0)$ can be performed, with fixed accuracy, by
means of a number of discretization points that grows very slowly as $v_0\to
0$. (In our implementations we typically use a number of discretization
points to evaluate this integral that remains constant for all required
small values of $v_0$.) The evaluation of the integral on the left of the
point $\alpha=\alpha^*(v_0,\rho_0)$ with fixed accuracy requires a number of
discretization points that does grow somewhat faster, as $v_0\to 0$, than
the one on the right, but, we have found in practice that the latter
integral can be obtained with fixed accuracy by means of a number of
discretization points that grows only logarithmically with $v_0$ as $v_0\to
0$.

To obtain an approximate expression for $\alpha^*(v_0,\rho_0)$ we note that,
since
\begin{equation}\label{Jalpha}
\tilde{J}(v_0,\rho_0,\alpha)=\alpha\frac{\sqrt{v_0+\rho_0\sin\alpha^2}}{\sin\alpha^2}\quad\mbox{for}\quad
\quad \alpha > \sqrt{ \arcsin( \frac{v_0}{\rho_0})},
\end{equation}
for $\alpha \ll 1$ and for sufficiently small values of $v_0$ (such that the
inequality constraint in~\eqref{Jalpha} is satisfied) we have $\alpha^2~\sim
\sin \alpha^2$, and thus letting $\eta = \frac{v_0}{\rho_0}$,
\begin{equation}\label{J_asym}
\tilde{J}(v_0,\rho_0,\alpha)\sim \sqrt{\rho_0}f_n(\alpha)\quad\mbox{where}\quad
    f_\eta(\alpha)=\frac{\sqrt{\eta + \alpha^2}}{\alpha}.
\end{equation}
It follows that, for a given constant $C>0$, the fixed-slope point
$\alpha_\eta$ for which $f_\eta'\left(\alpha_\eta\right)=-C$ is
approximately given as a root of the equation
\begin{equation}
-\frac{\eta}{\alpha_\eta^2\sqrt{\eta+\alpha_\eta^2}}=-C,\quad \mbox{or
 equivalently}, \quad \frac{\alpha_\eta^4}{\eta} +
 \frac{\alpha_\eta^6}{\eta^2} - \frac1{C^2}= 0.
\end{equation}
Clearly, thus, an approximation of the quantity $\alpha_\eta^2$ can be
obtained, in closed form, as a root of a certain polynomial of degree
three. A Taylor expansion of the resulting root as a function of $\eta$
around $\eta=0$ shows that
\begin{equation}
\alpha_\eta=O\left(\eta^\frac{1}{3}\right)\quad\mbox{as} \quad \eta
\rightarrow 0,
\end{equation}
and, since $\eta = \frac{v_0}{\rho_0}$, it follows that the constant slope point
$\alpha^*(v_0, \rho_0)$ is given, for each constant $\rho_0$, by
\begin{equation}
\alpha^*(v_0,\rho_0)\sim v_0^\frac{1}{3}\quad\mbox{as}\quad v_0\rightarrow 0.
\end{equation}.

Since the function $H = H(u_0,v_0,\rho,\theta)$ in
equation~\eqref{edge_polar} is modulated by a smooth windowing function that
is akin to the ``discontinuous window function'' $\tilde H_{\rho_0}$ in
equation~\eqref{BL_exp}, it is reasonable to expect that the inner
$\rho$-integral in~\eqref{edge_polar} gives rise to a $\alpha$-integrand
which develops a similarly behaved, albeit smoother, boundary-layer.  We
illustrate this in Figure~\ref{BLSmooth_Fig} (the left portion of which
should be compared to the right portion of Figure~\ref{BLRaw_Fig}), which
displays the function
\begin{equation}
J(v_0,\rho_0,\alpha)=\alpha \int_{-\frac{v_0}{\sin\alpha^2}}^\infty
H(v_0,\rho,\alpha^2)\frac{\rho_0\rho}{\sqrt{v_0+\rho_0\sin\alpha^2}},
\end{equation} 
where 
\begin{equation}\label{Hsmooth}
H_{\rho_0}(\rho,\theta)=W(\frac{\rho}{\rho_0}),\quad W(\rho)=\left\{ \begin{array}{ll} e^{-\frac{1}{1-\rho^2}}, & \rho < 1\\ 0, & \rho \geq
1.\end{array}\right.
\end{equation}
It follows that the $\theta$-integration strategy outlined above in this
section for the function $\tilde I$ (based on the change of variables
$\theta=\alpha^2$ and partitioning of integration intervals at the point
$\alpha =\alpha^*$) applies to the integrand given by the inner integral in
equation~\eqref{edge_polar}: this strategy is incorporated as part of our
algorithm, and is thus demonstrated in the numerical examples presented in
Section~\ref{sec_numerical}.

\begin{figure}[h]
\center
\includegraphics[scale=0.32]{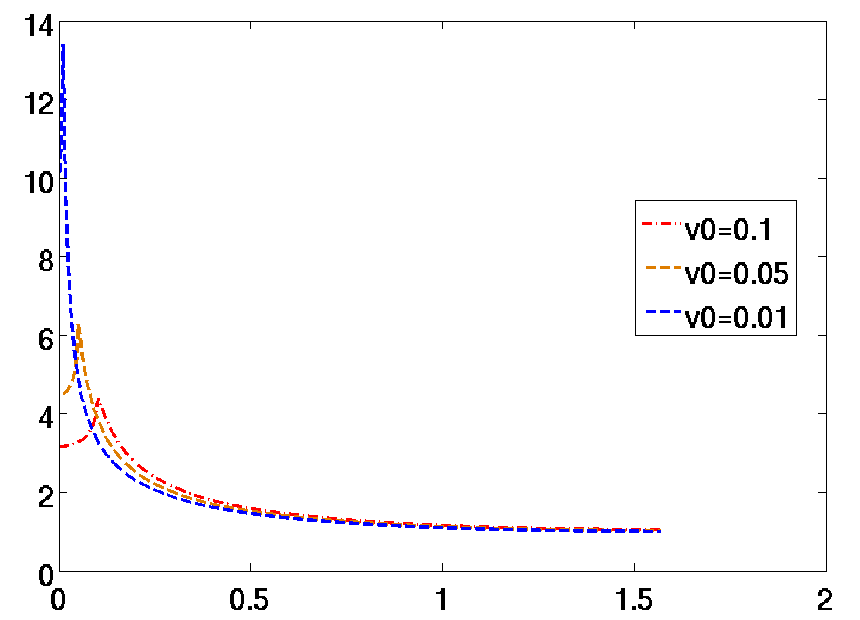}
\includegraphics[scale=0.32]{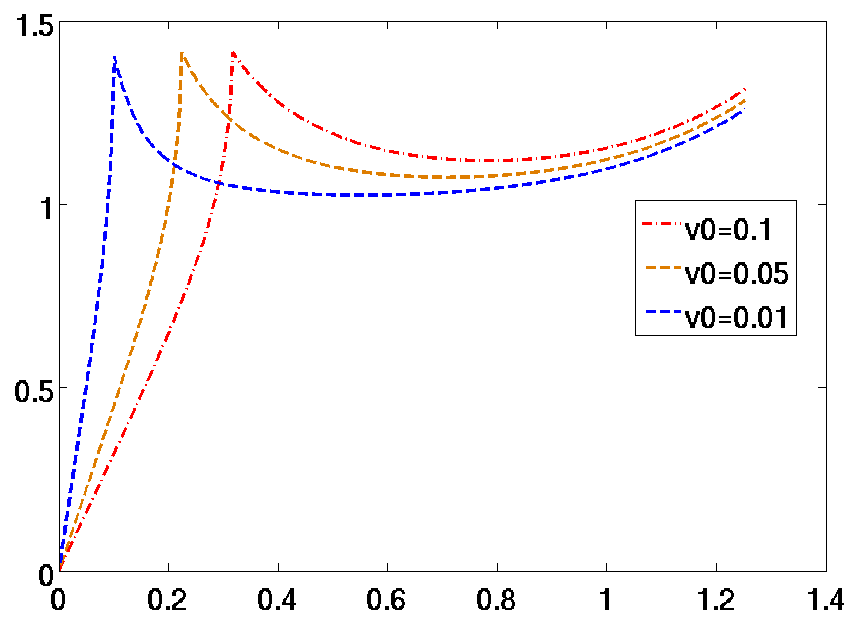}
\caption{Left: boundary layer for $\tilde{I}(v_0,\rho_0,\theta)$ on the
  interval $[0, \frac{\pi}{2}]$. Right: quadratic regularization
  $\tilde{J}(v_0,\rho_0,\alpha)=\alpha \tilde{I}(v_0,\rho_0,\alpha^2)$, where $\alpha\in[0,\sqrt{\frac{\pi}{2}}]$.  \label{BLRaw_Fig}}
\end{figure}

\begin{figure}[h]
\center
\includegraphics[scale=0.32]{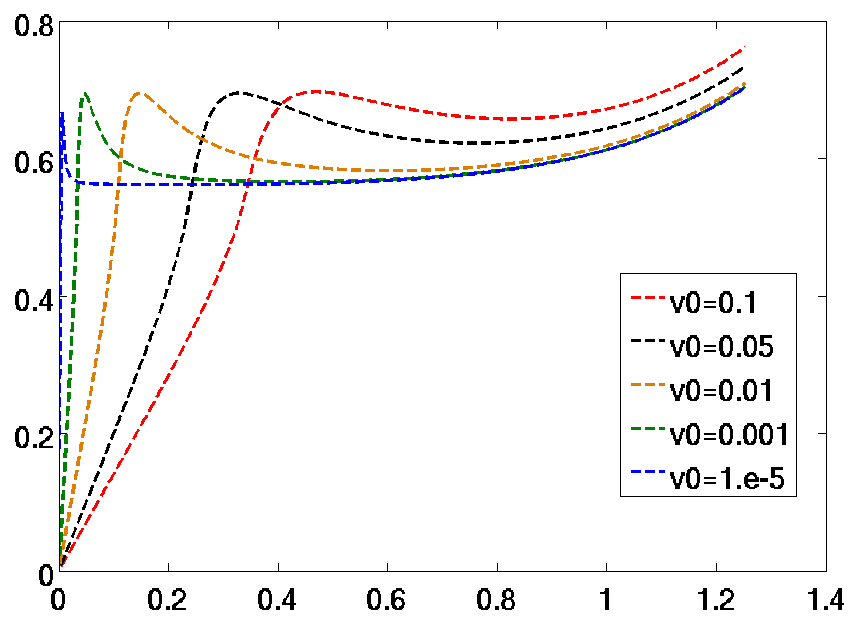}
\includegraphics[scale=0.32]{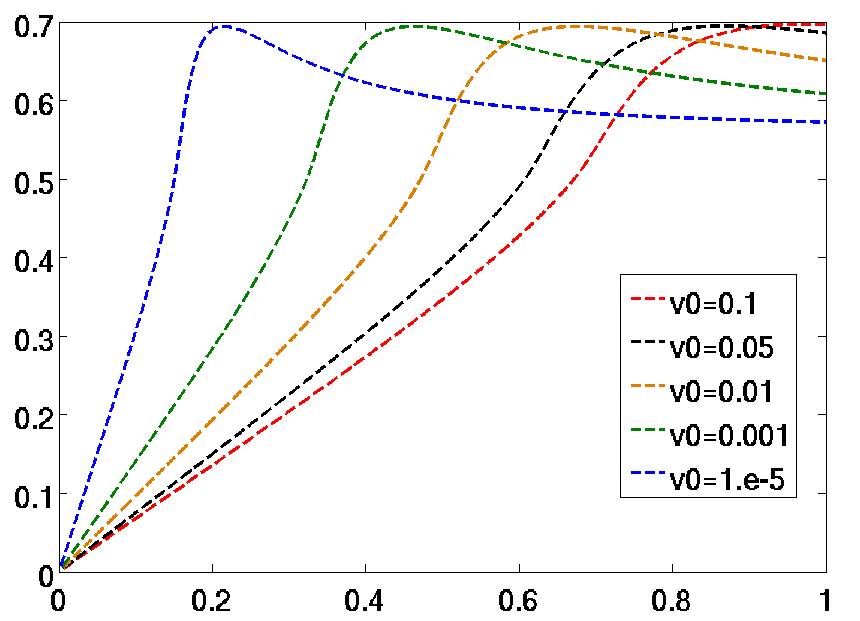}
\caption{Left: numerical values of $J(v_0,\rho_0,\alpha)$ for the smooth
  function $H_{\rho_0}(\rho,\theta)$ given in
  equation~\eqref{Hsmooth}. Right: normalized view on the interval
  $[0,\alpha^*(v_0)]$ for various values of $v_0$.  \label{BLSmooth_Fig}}
\end{figure}

\section{Babinet Principle for Acoustic Problems}\label{appendix_babinet}
As mentioned in Section~\ref{apertures}, for ease of reference, in this
appendix we present a derivation of the Babinet principle for scalar waves;
see also~\cite{BouwkampReview}. Let $\Gamma$ be an open (bounded) flat
screen which lies in the $z=0$ plane, and let $u^i$ and $u^s_{\Gamma}$ denote 
the incident wave (with sources contained in the
semi-space $z<0$), and the corresponding field scattered by $\Gamma$ under
the Dirichlet boundary condition given by
equation~\eqref{b_conds} with $f=-u^{i}_{|_\Gamma}$. The corresponding total
field is denoted as $u^T_\Gamma=u^i+u^s_\Gamma$.

Calling $\Gamma^c$ the complement of $\Gamma$ in the $z=0$ plane, let
$v^s_{\Gamma^c}$ denote the field scattered by $\Gamma^c$ under Neumann
boundary conditions, and let $v^T_{\Gamma^c}$ denote the corresponding total
field. We now establish the Babinet principle that relates the
Dirichlet-screen problem to the Neumann aperture problem, namely
\begin{equation}\label{babinet_form}
u^T_{\Gamma} + v^T_{\Gamma^c} = u^i\quad\mbox{for}\quad z>0,
\end{equation}
with an associated formula, given below, for $z<0$. The corresponding
Babinet principle relating the Neumann-Screen problem to the
Dirichlet-aperture problem follows similarly.

In the Dirichlet-screen/Neumann-aperture problem the total field
$v^T_{\Gamma^c}$ satisfies $v^T_{\Gamma^c} = u^i$ on $\Gamma$---since
$v^T_{\Gamma^c}=u^i + v^s_{\Gamma^c}$, and since $v^s_{\Gamma^c}$ is an odd
function of $z$ (as it equals a double-layer potential with source density
on $\Gamma^c$).  Thus, defining $w(x,y,z)=v^T_{\Gamma^c}(x,y,z)$ for $z>0$,
and $w(x,y,z)=v^T_{\Gamma^c}(x,y,-z)$ for $z<0$, we see that $w$ satisfies
the following properties
\begin{itemize}
\item[---] The boundary values of $w$ on $\Gamma$ satisfy $w_{|_\Gamma} =
 u^i_{|_\Gamma}$;
\item[---] $w$ is a radiative solution in $\mathbb{R}^3$ (since
  $v^T_{\Gamma^c} $ is radiating behind the screen)
\item[---] $w$ is continuous across $\Gamma^c$ (by definition) and its
  normal derivative across $\Gamma^c$ is continuous (since $v^T_{\Gamma^c}$
  satisfies homogeneous Neumann boundary conditions on $\Gamma^c$).
\end{itemize}
It follows that $w=-u^s_{\Gamma}$ (by uniqueness of solution to the
Dirichlet problem on $\Gamma$) and, therefore, that
$v^T_{\Gamma^c}=-u^s_\Gamma$ for $z>0$ or, in other words, that
$v^s_{\Gamma^c} = -u^s_\Gamma-u^i$ for $z>0$---and thus
equation~\eqref{babinet_form} follows. Since, as stated above,
$v^s_{\Gamma^c}$ is an odd function, its values in the region $z<0$ follow
by symmetry.

\bibliographystyle{plain}
\bibliography{./Screens}

\end{document}